\documentclass{article}

\usepackage{graphicx} 
\usepackage{amsmath}
\usepackage{amsthm}
\usepackage{amssymb}
\usepackage{authblk}
\usepackage{bm}
\usepackage{enumerate}
\usepackage[a4paper, total={6in, 8in}]{geometry}
\usepackage{hyperref} 
\usepackage{mathrsfs}
\allowdisplaybreaks

\title{Convergence of an IP DG Method for the Quad-Curl Problem}
\author{Xianhao ZENG \\ email \href{mailto:xhzeng6-c@my.cityu.edu.hk}{xhzeng6-c@my.cityu.edu.hk}}
\affil{Department of Mathematics, City University of Hong Kong, Kowloon, Hong Kong SAR, China}

\theoremstyle{plain}
\newtheorem{theorem}{\textbf{Theorem}}[section]
\newtheorem{lemma}{\textbf{Lemma}}[section]

\theoremstyle{definition}
\newtheorem{definition}{\textbf{Definition}}[section]
\newtheorem{assumption}{\textbf{Assumption}}[section]
\newtheorem{remark}{\textbf{Remark}}[section]

\numberwithin{equation}{section}
\allowdisplaybreaks[4]

\begin{document}

\maketitle

\begin{abstract}
    This work analyzes revises the interior penalty (IP) discontinuous Galerkin (DG) method imposed in [Chen, G., Qiu, W., \& Xu, L. (2021). Analysis of an interior penalty DG method for the quad-curl problem. IMA Journal of Numerical Analysis, 41(4), 2990-3023.] for the quad-curl problem in a nonconvex polyhedral domain, while introducing a piecewise constant coefficient matrix. We derive two main results: Under minimal regularity assumptions, we prove that the numerical solutions converge strongly to the true solution in the $ H(\text{curl}) \times H^1(\Omega)$ norm. Under slightly higher regularity, we establish the optimal estimate of the convergence rate depending on the regularity of the solution. These two results, serving as a complement to the existing literature, completely answer how the concerned IP DG method performs on quad-curl problems with low regularity.
\end{abstract}

\section{Introduction}
Let $ \Omega $ be an open bounded simply-connected (can be nonconvex) polyhedral Lipschitz domain in $ \mathbb{R}^3 $. This paper focuses on the following quad-curl problem: find the vector field $ \bm{u} $ and the Lagrange multiplier $p$ such that
\begin{equation}\label{strong form}
    \begin{aligned}
        \nabla\times\nabla\times A(\nabla\times\nabla\times \bm{u}) + \nabla p = \bm{f} \quad&\text{in } \Omega; \\
        \nabla\cdot \bm{u} = 0 \quad&\text{in } \Omega; \\
        \bm{n} \times \bm{u} =0; \quad \bm{n} \times\nabla\times \bm{u} = 0 \quad & \text{on } \partial\Omega; \\
        p = 0 \quad & \text{on } \partial\Omega.
    \end{aligned}
\end{equation}
Here $ \bm{n} $ stands for the outward unit normal vector on $ \partial\Omega $, and the source term $ \bm{f} \in L^2(\Omega)^3 $. In addition, we assume that
\begin{itemize}
    \item $ \{\Omega_i\}_{i\in\Lambda} $ is a partition of $\Omega$ into Lipschitz polyhedra;
    \item $ A \in [ L^\infty(\Omega)]^{3\times3}$ takes constant symmetric‑matrix values on each subdomain $\{\Omega_i\}$;
    \item there exists a uniform constant $\alpha>0$ such that for all $ x\in\Omega$, 
    \begin{equation*}
        \xi^T A(x) \xi \ge \alpha\|\xi\|_{l^2},\;\forall \xi \in \mathbb{R}^3.
    \end{equation*}
\end{itemize}
The main goal of the present work is to provide a thorough analysis of the performance of the interior penalty (IP) discontinuous Galerkin (DG) method for the quad-curl problem as imposed in \cite{chen2021analysis}, under a very low regularity assumption. We present our discovers as two main conclusions, both of which are novel in the literature.
\subsection{Background in Physics}
The quad-curl problem originates from several physical applications, including Maxwell transmission eigenvalue theory (cf. \cite{monk2012finite}, \cite{haddar2004interior}) and resistive magnetohydrodynamics (MHD) (cf. \cite{zheng2011nonconforming}). In the context of the two models, we clarify the physical interpretation of the coefficient matrix $ A $.
\begin{itemize}
    \item In the inverse electromagnetic scattering theory, the transmission eigenvalue problem for the anisotropic Maxwell equations can be formulated in the following fourth-order problem: find the vector field $ \bm{u} $ and the number $k$ such that
    \begin{align*}
        ( \nabla\times\nabla\times - k^2 N )(N - I)^{-1}(\nabla\times\nabla\times \bm{u} - k^2 \bm{u}) = 0 \quad & \text{in } \Omega; \\
        \bm{n} \times \bm{u} = 0; \quad \bm{n} \times\nabla\times \bm{u} = 0 \quad & \text{on } \partial\Omega;
    \end{align*}
    where $ N $ stands for the index of refraction of an anisotropic medium, such that $ N $, $ N^{-1} $, together with either $ (N - I)^{-1} $ or $ (I - N)^{-1} $ are bounded positive definite. 
    \item The resistive MHD system reads: finding the velocity $ \bm{u} $, the pressure $ p $ and the magnetic induction field $B$ such that
    \begin{align*}
        \rho (\bm{u}_t + (\bm{u} \cdot \nabla)\bm{u}) + \nabla p & = \frac{1}{\mu_0} (\nabla \times B) + \mu \Delta \bm{u} \quad \text{in} \, \Omega, \\
        B_t - \nabla \times (\bm{u} \times B) & = -\frac{\eta}{\mu_0} (\nabla \times)^2 B \\
        &\quad - \frac{d_i}{\mu_0} \nabla \times ((\nabla \times B) \times B) - \frac{\eta^2}{\mu_0} (\nabla \times)^4 B \quad \text{in} \, \Omega, \\
        \nabla \cdot u & = 0 \quad \text{in} \, \Omega, \\
        \nabla \cdot B & = 0 \quad \text{in} \, \Omega.
    \end{align*}
    where $ \eta $ is the resistivity, $ \eta_2 $ is the hyper-resistivity, $ \mu _0 $ is the magnetic permeability of free space, and $ \mu $ is the viscosity.
\end{itemize}
Whenever such problems are posed on heterogeneous (multi‑material) domains, the contrast in the physical properties necessarily causes discontinuity of $ A $. This is an essential feature of the problem and requires careful consideration in numerical analysis. Motivated by this, we adopt the current setting for $\Omega_i$ and $ A $ as in \ref{strong form}. This setting is sufficiently simple while covering the most relevant scenarios encountered in practice. 
\subsection{Literature Review}
Compared with the vast work on numerical approaches to the MHD problem \textit{without} the quad-curl term, the amount of studies that involve the quad-curl term is limited (cf. \cite{zheng2011nonconforming}), and even fewer papers consider the discontinuous coefficient $ A $. The main reasons for this situation include the following essential difficulties associated with this problem:
\begin{enumerate}
    \item \textit{Lack of regularity.} At the continuous level of PDEs, the regularity of the solution depends strongly on the geometric properties of the domain. When $ A = I $, it is known (cf. \cite{nicaise2018singularities}) that for a smooth domain the exact solution belongs to $ [H^4(\Omega)]^3 $; however, when the domain has point or edge singularities –- which a nonconvex polyhedral domain typically exhibits –- the solution generally does not lie in $ [H^3(\Omega)]^3 $. For a nonconvex domain, the regularity result is proved in \cite{chen2021analysis}: $ \bm{u} \in [H^{ \frac{1}{2} + \delta }(\Omega)]^3 $, $ \nabla\times \bm{u} \in [H^1_0(\Omega)]^3 $, with the estimate
    $$ \| \bm{u} \|_{H^{ \frac{1}{2} + \delta }(\Omega)} + \| \nabla\times \bm{u} \|_{H^{ 1 }(\Omega)} + \| (\nabla\times)^2 \bm{u} \|_{L^2(\Omega)} + \| (\nabla\times)^4 \bm{u} \|_{L^2(\Omega)} + \| \nabla p \|_{L^2(\Omega)}  \le C \| \bm{f} \|_{L^2(\Omega)}. $$
    This result appears to be optimal among the existing literature. \par
    Another source of regularity reduction is the discontinuity of $A$: if $ A (\nabla\times)^2 \bm{u} $ were to gain continuity, $ (\nabla\times)^2 \bm{u} $ would lose continuity instead. \par 
    In conclusion, the regularity of $ (\nabla\times)^2 \bm{u} $ is so low -- barely better than $L^2$ and by no means above $ H^\frac{1}{2} $ -- that the conventional trace theorem fails to apply, rendering many traditional techniques in DG methods ineffective.
    \item \textit{The cost of conformity.} Choosing a suitable finite element space for discretization poses a significant challenge. \par
    First, $ H^1 $-conforming (or $C^0$-conforming) elements, although popular due to their simple implementation and good approximation properties, can cause serious troubles when dealing with curl‑related problems. As figured out in \cite{costabel1991coercive} and \cite{guermond2003mixed}, $ H^1(\Omega) \cap H_0(\text{curl};\Omega) $ is not dense in $ H_0(\text{curl};\Omega) \cap H(\text{div};\Omega) $, whenever the magnetic permeability $ \mu $ and conductor‑vacuum interface $ \Sigma $ are simultaneously non-smooth, or the domain $ \Omega $ is nonconvex. Consequently, when the Lagrange element is applied to solve the MHD flow, the numerical solutions for the magnetic field $ \mathbf{H}_h $ fail to converge to the exact solution $ \mathbf{H} $ in general. We can expect the $ H^2 $-conforming elements to suffer only more when applied on the quad-curl problem. Indeed, the numerical approach in \cite{zhang2009family} which uses an $ H^2 $-conforming (or $C^1$-conforming) elements for the quad-curl problem converges to an $H^2$ projection of the exact solution. \par
    Furthermore, $H(\text{curl}^2)$-conforming elements are theoretically the most natural choice. Indeed, quad-curl problems in two dimensions have been successfully solved with $H(\text{curl}^2)$-conforming elements (cf. \cite{zhang2019h}). The situation in three dimensions, however, is different: constructing such families of finite elements still faces significant technical difficulties. To the author’s knowledge, the only attempt in the existing literature to build $H(\text{curl}^2)$-conforming elements on tetrahedra is \cite{zhang2020family}. That work produced elements of order at least 7, with at least 315 degrees of freedom per cell, and proved interpolation error estimates under the assumption that $ \bm{u}, \nabla \times \bm{u} \in H^{\frac{7}{2} + \delta }(\Omega) $. Nevertheless, these elements fail to preserve $H(\text{curl}^2)$‑conformity under a general affine mapping. This shortcoming prevents their widespread use for the general quad-curl problem, both in theory and in practice. \par
    We now turn to $H(\text{curl})$-conforming elements. These elements feature a straightforward, low‑cost construction and possess desirable commuting and approximation properties (cf. \cite{nedelec1980mixed}, \cite{nedelec1986new}, \cite{alonso1999optimal}, \cite{monk2003finite} and \cite{boffi2013mixed}). Moreover, several numerical methods based on $H(\text{curl})$‑conforming elements have been successfully applied to the quad‑curl problem. A nonconforming method using the first kind N{\'e}d{\'e}lec family was studied in \cite{zheng2011nonconforming} under the regularity assumption $ \bm{u} \in [H^4(\Omega)]^3 $. A discontinuous Galerkin (DG) method using $H(\text{curl})$-conforming elements was investigated in \cite{hong2012discontinuous}, requiring $ \bm{u}, \nabla\times \bm{u} \in [H^2(\Omega)]^3 $. An interior penalty (IP) DG method was introduced and analyzed in \cite{sun2016mixed}, requiring $ \bm{u}, \nabla\times \bm{u} \in [H^3(\Omega)]^3 $. In \cite{chen2021analysis}, another IP method was proposed and an error estimate was proved under regularity assumptions including $ (\nabla\times)^2 \bm{u} \in [H^{ \frac{1}{2} + \delta }(\Omega)]^3 $. \par
    However, the performance of such methods for problems with very low regularity remains unknown.
\end{enumerate}
Motivated by \cite{chen2021analysis}, we continue the analysis on the same IP DG method for this question under minimal or low regularity assumption, applying distinct techniques, however. The success of this work will illustrate the broad applicability of this method to the quad-curl problems.
\subsection{The Content of This Work}
In the present work, we consider two scenarios for the exact solution $\bm{u}$:
\begin{itemize}
    \item \textit{No extra regularity}: $ (\nabla\times)^2 \bm{u} \in [L^2(\Omega)]^3 $, which is the minimum required for the weak formulation; or
    \item \textit{Low regularity}: $ (\nabla\times)^2 \bm{u} \in [H^s(\Omega)]^3 $ for $ 0 \le s \le \frac{1}{2} $, a scenario that remains novel in the literature.
\end{itemize}
For these two settings, we derive two convergence results, respectively.
\begin{enumerate}
    \item Under no extra regularity, we prove that the numerical solutions $ \{(\bm{u}_h,p_h)\}_{ h \in \mathcal{H} } $ converge to the weak solution to \ref{weak form}, $ (\bm{u},p) $, strongly in $ H_0(\text{curl}; \Omega) \times H^1_0(\Omega) $ (i.e., by norm). This convergence is theoretical, however, as no convergence rate is obtained.
    \item Under a set of mild assumptions (Assumption \ref{assumption}), including $ (\nabla\times)^2 \bm{u} \in [H^s(\Omega)]^3 $ for $ 0 \le s \le \frac{1}{2} $, we establish the optimal estimate: we follow the framework of \cite{ern2022quasi}, employing an augmented norm $ \| \cdot \|_{E_\#} $ (Definition \ref{aug norm}) which is stronger than the energy norm but compatible with the regularity, such that the error $ \| \bm{u}_h - \bm{u} \|_{E_\#} $ gains a convergence rate based on the regularity of the exact solution.
\end{enumerate}
To derive these two results, we employ different techniques. \par 
The theoretical convergence is obtained via a compactness argument introduced in \cite{qiu2023enriched} for the mixed formulation of the biharmonic equation. Applying it to the current method, we have successfully overcome the following technical difficulties:
\begin{enumerate}
    \item The compactness properties on embedding operators defined on spaces involving curl and div exhibit unique and subtle patterns. Based on a careful investigation on them, we reveal the convergence utilizing the technique of \textit{Hodge mapping} as in \cite{hiptmair2002finite}, \cite{monk2003finite} and the $H^1$ \textit{averaging operator} as in \cite{ern2017finite}. Notably, the proof never assumes the existence of any $H(\text{curl}^2)$-conforming element spaces on the mesh, nor does it require any interpolations onto such spaces. 
    \item The piecewise constant coefficient $ A $ admits discontinuities inside the domain, giving rise to the following type of pairing in our proof:
    $$ \langle  \bm{n}_{ F } \times \{\!\{  \nabla \times \bm{u}_h \}\!\}, [\![ A (\nabla \times)^2 \bm{v} ]\!] \rangle, $$
    which appears in the literature for the first time. 
    \item We allow non‑matching meshes for the subdomains $ \{ \Omega_i\} $, making the argument applicable in general settings. The cost for this is the need for handling terms defined on the subdomain boundaries $ \{ \partial \Omega_i\} $, which is also novel in literature.
\end{enumerate}
By addressing these challenges, we have demonstrated the great potential of this scheme in yielding convergence results of a broad class of numerical methods for various PDEs under no extra regularity assumption. \par
The optimal convergence rate is derived from the quasi‑optimal argument that is introduced in \cite{ern2022quasi}. Using a face‑to‑cell lifting operator, the author established an alternative version of Strang’s lemma that extends the regularity assumption to barely above $H^1$, and proved error estimates for several nonconforming methods applied to elliptic PDEs with contrasting coefficients. Specifically, \cite{ern2022quasi} implements this scheme by establishing the following techniques:
\begin{enumerate}
    \item The quasi-optimal property of the numerical solution, which means that the numerical solution is, up to a generic constant, as good as the best approximation in the finite element space. The framework has been raised in \cite{veeser2018quasiI} and \cite{veeser2018quasiIII}, and is continued in \cite{ern2022quasi}.
    \item The face-to-cell lifting operator, which utilizes subtle properties of the trace operator. \cite{ern2022quasi} constructs the operator in the Sobolev–Slobodeckij norm induced fractional-order Sobolev spaces, which is the first unambiguous clarification in literature. See the references therein for detailed information.
    \item The bilinear form which extends the notion of face integrals by using the face-to-cell operator. This technique is the bridge between the analysis of the exact solution and that of a finite element function, which appears crucial in the study. \cite{ern2022quasi} has proposed a rigorous argument on this.
\end{enumerate}
This strategy has been employed in the study of various fourth-order problems, such as \cite{dong2022hybrid} and \cite{dong2024c}, which analyze the HHO method for the biharmonic equation. In the present work, we follow the approach of \cite{dong2024c} and derive a similar error estimate for the quad curl problem, but state the regularity assumption in terms of Hilbert space indices. \par
However, unlike the aforementioned articles, the quasi-optimal argument does not conclude the proof, because it still requires certain novel nontrivial approximation properties of N{\'e}d{\'e}lec family. We establish these approximation properties for second, third and fourth ordered differential operators in Theorem \ref{Nedelec est 3}. \par
Moreover, the following convergence pattern is implied in this conclusion: The augmented norm $ \| \cdot \|_{E_\#} $ (Definition \ref{aug norm}) enriches the energy norm with properly scaled quantities related to $ \bm{u} $ computed on each cell. Therefore, for a given point, the contribution from the cells containing the point to the error-which describe the convergence rate near that point-depends only on the regularity of $\bm{u}$ in a neighborhood of that point. \par
The rest of this paper is organized as follows: Section 2 describes in detail the continuous PDE setting, the $H(\text{curl})$‑conforming IP method, and the finite element space. In Section 3, we apply the compactness argument to derive unconditional convergence. In Section 4, we employ the quasi‑optimal argument to obtain convergence rates. 

\section{Problem Setting}
\subsection{Basic Notations}
Throughout this work, $ C > 0 $ will stand for a finite constant which might depend on the geometry of the domain $ \Omega $, the subdomains $ \{ \Omega_i \} $, and the evaluation of $ A $, but shall be independent of the construction of mesh $ \{\mathcal{T}_h\} $, the mesh size $ h $, and the specific selection of a function from a certain function space. \par
We apply conventional notation for the Lebesgue and Sobolev spaces. Specifically, for $ 1 \le p < \infty $, the integer-order Sobolev space $ W^{m,p}(\Omega) $ is the Banach space induced by the norm
$$ \| u \|_{ W^{m,p}(\Omega)} : = ( \sum_{ |\alpha| \le m } \| \partial_\alpha u \|_{ L^p(\Omega)}^p )^{\frac{1}{p}}, $$
while the Sobolev–Slobodeckij seminorm is derived from the double integral
\begin{equation*}
    |u|_{W^{\sigma,p}(\Omega)}=(\int_\Omega\int_\Omega \frac{|u(x)-u(y)|^p}{|x-y|^{d + \sigma p }} dx dy )^{\frac{1}{p}}, \quad 0 < \sigma < 1,
\end{equation*}
and for $ s > 0 $, the fractional-order Sobolev space $ W^{s,p}(\Omega) $ is derived from the norm
$$ \| u \|_{ W^{s,p}(\Omega)} : = ( \| u \|_{ W^{s,p}(\Omega)}^p +  \sum_{ |\alpha| \le m } | \partial_\alpha u |_{W^{\sigma,p}(\Omega)}^p )^{\frac{1}{p}}, $$
where $ m = [s] $ is the integer part, and $ \sigma = s - m $. We denote by $ (\cdot, \cdot)_{\Omega} $ the standard inner product in $ L^2(\Omega) $. \par
We use bold letters to denote an $\mathbb{R}^3$ vector field. For example, $ \bm{u} := (u_1, u_2, u_3) $. The norm $ \| \bm{u} \|_{W^{s,p}(\Omega)} $ is defined component-wisely:
$$ \| \bm{u} \|_{W^{s,p}(\Omega)} := (\| u_1 \|_{W^{s,p}(\Omega)}^p + \| u_2 \|_{W^{s,p}(\Omega)}^p + \| u_3 \|_{W^{s,p}(\Omega)}^p)^\frac{1}{p}, $$
and similarly for the inner product $ (\bm{u}, \bm{v})_{\Omega} $. \par
The Hilbert spaces involving div and curl will be the foundation of the present work. They are
\begin{align*}
    H(\text{div};\Omega) & := \{ \bm{u} \in [L^2(\Omega)]^3| \nabla \cdot \bm{u} \in L^2(\Omega) \}; \\
    H(\text{curl};\Omega) & := \{ \bm{u} \in [L^2(\Omega)]^3| \nabla \times \bm{u} \in [L^2(\Omega)]^3 \}; \\
    H(\text{curl}^2;\Omega) & := \{ \bm{u} \in H(\text{curl};\Omega)| \nabla \times \nabla \times \bm{u} \in [L^2(\Omega)]^3 \};
\end{align*}
with the norms
\begin{align*}
    \| \bm{u} \|_{H(\text{div};\Omega)} & = ( \| \bm{u} \|_{L^2(\Omega)}^2 + \| \nabla \cdot \bm{u} \|_{L^2(\Omega)}^2 )^{\frac{1}{2}}; \\
    \| \bm{u} \|_{H(\text{curl};\Omega)} & = ( \| \bm{u} \|_{L^2(\Omega)}^2 + \| \nabla \times \bm{u} \|_{L^2(\Omega)}^2 )^{\frac{1}{2}}; \\
    \| \bm{u} \|_{H(\text{curl}^2;\Omega)} & = ( \| \bm{u} \|_{L^2(\Omega)}^2 + \| \nabla \times \bm{u} \|_{L^2(\Omega)}^2 + \| \nabla \times \nabla \times \bm{u} \|_{L^2(\Omega)}^2 )^{\frac{1}{2}}. \\
\end{align*}
We also define
\begin{align*}
    H(\text{div }0 ;\Omega) & := \{ \bm{u} \in [L^2(\Omega)]^3 | \nabla \cdot \bm{u} = 0 \}; \\
    H(\text{curl }0 ;\Omega) & := \{ \bm{u} \in [L^2(\Omega)]^3 | \nabla \times \bm{u} = 0 \};
\end{align*}
and
\begin{align*}
    H_0(\text{div};\Omega) & := \{ \bm{u} \in H(\text{div};\Omega) | \bm{n} \cdot \bm{u} = 0 \text{ on } \partial\Omega \}; \\
    H_0(\text{curl};\Omega) & := \{ \bm{u} \in H(\text{curl};\Omega) | \bm{n} \times \bm{u} = 0 \text{ on } \partial\Omega \}; \\
    H_0(\text{curl}^2;\Omega) & := \{ \bm{u} \in H(\text{curl}^2;\Omega) | \bm{n} \times \bm{u} = \bm{n} \times \nabla \times \bm{u} = 0 \text{ on } \partial\Omega \}. \\
\end{align*}

\subsection{The Weak Formulation}
We study the following weak formulation for \ref{strong form}: Find $ \bm{u} \in H_0(\text{curl}^2;\Omega) $ and $ p \in H^1_0(\Omega) $ such that
\begin{equation}\label{weak form}
    \begin{aligned}
        (A(\nabla\times\nabla\times \bm{u} ),\nabla\times\nabla\times \bm{v})+(\nabla p,\bm{v})=(\bm{f},\bm{v}),\quad & \forall \bm{u}\in H_0(\text{curl}^2;\Omega);\\
        (\bm{u},\nabla q)=0,\quad&\forall q\in H^1_0(\Omega).
    \end{aligned}
\end{equation}
\begin{theorem}
    \ref{weak form} is well posed, with the solution satisfying 
    \begin{enumerate}
        \item $ \bm{u} \in H^{s_0}(\Omega) $, $ \exists \frac{1}{2} < s_0 < 1 $;
        \item $ \nabla \times \bm{u} \in H^1_0(\Omega) $;
        \item $ \| \bm{u} \|_{H^{s_0}(\Omega)} + \| \nabla \times \bm{u} \|_{H^1(\Omega)} + \| (\nabla \times)^2 \bm{u} \|_{L^2(\Omega)} + \| (\nabla \times)^2 A (\nabla \times)^2 \bm{u} \|_{L^2(\Omega)} + \| p \|_{L^2(\Omega)} \le C \| \bm{f} \|_{L^2(\Omega)} $.
    \end{enumerate}
\end{theorem}
The proof is almost the same as in \cite[Theorem 2.4]{chen2021analysis}: just note that the setting of $ A $ perceives coercivity. 

\subsection{The Discrete Formulation}
Let $\{\mathcal{T}_h\}_{h\in\mathcal{H}}$ be a sequence of quasi-uniform shape regular meshes into tetrahedra $ K \in \mathcal{T}_h $. We do \textbf{not} assume that $ \mathcal{T}_h $ is subordinate to subdomains $ \{\Omega_i\} $, i.e., each cell $ K $ lies in exactly one of those $ \Omega_i $, except for section 4. \par
Let $ \mathcal{F}^\textit{int}_h $ and $ \mathcal{F}^\partial_h $ be the collection of all internal interfaces and all boundary faces of $ \mathcal{T}_h $, respectively, and $ \mathcal{F}_h $ for both kinds. We specify a uniform orientation for every $ K\in\mathcal{T}_h $, such that each interface $F$ admits the representation $ F = \partial K^+ \cap \partial K^- $, and $  \bm{n}_F :=  \bm{n}_{\partial K^+}|_F $. Let $ \mathcal{T}_F=\{K^+,K^-\} $ for the interfaces and $ \mathcal{T}_F=\{K^+\} $ on the boundary. Let $ \epsilon_{K,F}= \bm{n}_K\cdot  \bm{n}_F $. We assign to any piecewise smooth function $ \phi $ the quantities defined on $ \mathcal{F}_h $:
\begin{equation*}
    \{\!\{ \phi \}\!\}_F:=
    \begin{cases}
        \frac{1}{2}(\phi|_{K^+}+\phi|_{K^-})|_F,\;F \in \mathcal{F}^\textit{int}_h \\
        \phi_F,\;F \in \mathcal{F}^\partial_h \\
    \end{cases};
    [\![ \phi ]\!]_F:=
    \begin{cases}
        \phi|_{K^+}-\phi|_{K^-},\;F \in \mathcal{F}^\textit{int}_h \\
        \phi_F,\;F \in \mathcal{F}^\partial_h \\
    \end{cases}.
\end{equation*}
We denote by $ \mathcal{P}_k(K)$ the space of polynomial functions up to order $ k $ on the cell $K$, and $ \mathcal{P}_k(\mathcal{T}_h) := \prod_{ K \in \mathcal{T}_h } \mathcal{P}_k(K) $ to be the broken polynomial space. For for \ref{weak form}, we study the IP DG method the same as the one in \cite{chen2021analysis}:
\begin{definition}
    \begin{enumerate} 
        \item Define the finite element spaces as
        \begin{equation*}
            E_h=H_0(\text{curl};\Omega)\cap[\mathcal{P}_k(\mathcal{T}_h)]^3; \quad Q_h=H^1_0(\Omega)\cap\mathcal{P}_{k+1}(\mathcal{T}_h); \quad k \ge 2. 
        \end{equation*}
        \item Define the bilinear forms $ a_h: E_h \times E_h \to \mathbb{R} $ as
        \begin{equation*}
            \begin{split}
                a_h( \bm{u}_h , \bm{v}_h ) = \sum_{K\in \mathcal{T}_h} ( A (\nabla\times)^2 \bm{u}_h , (\nabla\times)^2 \bm{v}_h )_{K} - \sum_{F \in \mathcal{F}_h} \langle \{\!\{ A (\nabla\times)^2 \bm{u}_h \}\!\},  \bm{n}_F \times [\![ \nabla\times \bm{v}_h ]\!] \rangle_{F}
                \\ - \sum_{F \in \mathcal{F}_h} \langle \{\!\{ A (\nabla\times)^2 \bm{v}_h \}\!\},  \bm{n}_F \times [\![\nabla\times \bm{u}_h ]\!] \rangle_{F} + \sum_{F \in \mathcal{F}_h} \frac{\tau}{h_F} \langle  \bm{n}_F \times [\![ \nabla\times \bm{u}_h ]\!],  \bm{n}_F \times [\![ \nabla\times \bm{v}_h ]\!] \rangle_{F}.
            \end{split}
        \end{equation*}
        where $ \tau > 0 $ is sufficiently large; $ b_h: Q_h \times E_h \to \mathbb{R} $ as
        \begin{equation*}
            b_h(q_h, \bm{v}_h )=\sum_{K\in \mathcal{T}_h}(\nabla q_h, \bm{v}_h )_K.
        \end{equation*}
        \item The discrete formulation is: Find $ ( \bm{u}_h ,p_h) \in E_h \times Q_h $ such that
        \begin{equation}\label{ip mtd}
            \begin{aligned}
                a_h( \bm{u}_h , \bm{v}_h )+b_h(p_h, \bm{v}_h )=(f, \bm{v}_h )_\Omega \quad & \forall \bm{v}_h \in E_h;\\
                b_h(q_h, \bm{u}_h )=0,\quad&\forall q_h\in Q_h.
            \end{aligned}
        \end{equation}
    \end{enumerate}
\end{definition}
Given this discrete formulation, the discrete energy norm on $E_h$ is derived as
\begin{equation*}
    \begin{aligned}
        \| \bm{v}_h \|_{E_h}:=(\sum_{K\in \mathcal{T}_h}\| \bm{u}_h \|_{L^2(K)}^2+\|\nabla\times \bm{u}_h \|_{L^2(K)}^2+\|\nabla\times\nabla\times \bm{u}_h \|_{L^2(K)}^2\\
        +\sum_{F \in \mathcal{F}_h}\frac{\tau}{h_F}\| \bm{n}_F\times[\![\nabla\times  \bm{v}_h ]\!]\|_{L^2(F)}^2)^{1/2}.
    \end{aligned}
\end{equation*}
We denote $ V_h := \{ \bm{v}_h \in E_h | b_h(q_h, \bm{v}_h )=0,\;\forall q_h\in Q_h\} $. As in (4.7) of \cite{chen2021analysis},
\begin{equation}\label{energy estimate}
    \| \bm{v}_h \|_{E_h}^2 \le C \sum_{K\in \mathcal{T}_h} \| \nabla\times\nabla\times \bm{u}_h \|_{L^2(K)}^2 + \sum_{F \in \mathcal{F}_h}\frac{\tau}{h_F}\| \bm{n}_F\times[\![\nabla\times  \bm{v}_h ]\!]\|_{L^2(F)}^2, \quad \forall \bm{v}_h \in V_h.
\end{equation}
\begin{theorem}\label{well-posed}
    For $ \tau > 0 $ sufficiently large, the numerical solution to \ref{ip mtd}, $ ( \bm{u}_h ,p_h) $, exists uniquely, with the energy estimates 
    $$ \| \bm{v}_h \|_{E_h} \le C \| \bm{f} \|_{L^2(\Omega)}; \quad \| p_h \|_{H^1_0(\Omega)} \le C \| \bm{f} \|_{L^2(\Omega)}. $$
\end{theorem}
The proof is similar to that of \cite[Theorem 4.5]{chen2021analysis}.

\subsection{Remarks on the Finite Element Space}
Before analyzing the $ H (\text{curl}) $ conforming method, it is necessary that we make some remarks on the properties of these finite elements. N{\'e}d{\'e}lec has constructed two families of $ H (\text{curl}) $ conforming elements on tetrahedra in \cite{nedelec1980mixed} and \cite{nedelec1986new}, now known as the first and second kind N{\'e}d{\'e}lec family, respectively. Analysis on these finite elements has been carried out in \cite{alonso1999optimal}, \cite{monk2003finite}, and \cite{boffi2013mixed}, etc. Specifically, the finite element space for the first kind N{\'e}d{\'e}lec family is
$$ N_k(K)=[\mathcal{P}_k(K)]^3 \oplus S_{k+1} \text{ where } S_{k+1} := \{ \bm{p} \in \mathcal{P}^H_{k+1}(K) | \overset{\to}{x} \cdot \bm{p} = 0 \}, \; k\ge 0 $$
(or $\mathcal{R}_{k+1}$ in some literature); while that of the second kind N{\'e}d{\'e}lec family is $ [\mathcal{P}_k(K)]^3 $, $ k \ge 1 $. Nevertheless, the interpolation operators onto each space, $ \Pi^{\text{curl; I}}_{k} $ and $ \Pi^{\text{curl; II}}_{k} $, satisfy:
\begin{theorem}
    $ \forall \bm{u} \in H(\text{curl};K) $ such that both $ \bm{u} $ and $ \nabla \times \bm{u} $ belong to $ [H^s(K)]^3 $, $ \frac{1}{2} < s \le k $, 
    \begin{equation}\label{Nedelec est 1}
        \begin{aligned}
            \| \bm{u} - \Pi^{\text{curl; I}}_{k} (\bm{u}) \|_{L^2(K)} & \le C h_K^s ( \| \bm{u} \|_{H^s(K)} + h_K \| \nabla \times \bm{u} \|_{H^s(K)}) ; \\
            \| \nabla \times \bm{u} - \nabla \times \Pi^{\text{curl; I}}_{k} (\bm{u}) \|_{L^2(K)} & \le C h_K^s \| \nabla \times \bm{u} \|_{H^s(K)}. \\
        \end{aligned}
    \end{equation}
    The estimates are also valid for $ \Pi^{\text{curl; II}}_{k} $
\end{theorem}
Since we will be applying the second kind exclusively in this work, we shall omit the subscript II, specifying hereafter that $ \Pi^{\text{curl}}_{k} $ stands for the second kind N{\'e}d{\'e}lec interpolation. In addition, $ \Pi^{\text{curl}}_{h,k} $ will be this interpolation implemented cell-wisely on $ \mathcal{T}_h $. \par
In addition, the lines in (5.9), (5.14) and (5.12) in \cite{chen2021analysis} imply that the following estimates of higher ordered terms hold:
\begin{lemma}\label{Nedelec est 2}
    \begin{enumerate}
        \item $ \forall \bm{u} \in H(\text{curl};K) $ such that $ \nabla \times \bm{u} \in [H^{t_1}(K)]^3 $, $ 1 \le t_1 \le k+1 $,
        \begin{align*}
            \| (\nabla \times)^2 (\bm{u} -\Pi^{\text{curl}}_{h,k} (\bm{u})) \|_{L^2(K)} & \le C h_K^{t_1-1} \| \nabla \times \bm{u} \|_{H^{t_1}(K)}; \\
            h_F^{\frac{1}{2}} \|  \bm{n}_F \times [\![\nabla\times (\bm{u} -\Pi^{\text{curl}}_{h,k} (\bm{u})) ]\!]  \|_{L^2(F)} & \le C h_K^{t_1-1} \| \nabla \times \bm{u} \|_{H^{t_1}(K)}. \\
        \end{align*}
        \item In addition, if we also have $ (\nabla \times)^2 \bm{u} \in [H^{t_2}(K)]^3 $, $ \frac{1}{2} < t_2 \le k $, then
        \begin{equation*}
            h_F^{\frac{1}{2}} \| \{\!\{A(\nabla\times)^2 (\bm{u} -\Pi^{\text{curl}}_{h,k} (\bm{u}))\}\!\} \|_{L^2(F)}  \le C ( h_K^{t_1-1} \| \nabla \times \bm{u} \|_{H^{t_1-1}(K)} + h_K^{t_2} \| (\nabla \times)^2 \bm{u} \|_{H^{t_2}(K)}  ).
        \end{equation*}
    \end{enumerate}
\end{lemma}
Next, as we shall see, the properties of the $ H (\text{div}) $ conforming finite elements will also appear vital. Constructions of such finite elements include the Raviart-Thomas (RT) elements and the Brezzi–Douglas–Marini (BDM) elements. The finite element space for RT elements is 
$$ RT_k(K)=[\mathcal{P}_k(K)]^3 \oplus \overset{\to}{x} \cdot \mathcal{P}^H_k(K), k \ge 0 $$
(or $\mathcal{D}_{k+1}$ in some literature); while that of BDM elements is $[\mathcal{P}_k(K)^3]$, $ k \ge 1 $. Research on these spaces can be found in \cite{monk2003finite}, \cite{boffi2013mixed}, etc. Specifically,
\begin{theorem}
\begin{enumerate}
    \item Denote by $\Pi^{RT}_k$ the Raviart-Thomas interpolation onto $RT_k(K)$. Then for any $ r \in (\frac{1}{2},k+1] $,
    \begin{equation}\label{RT est}
        \|(I-\Pi^{RT}_k)v\|_{L^2(K)}\le Ch_K^r|v|_{H^r(K)},\;\forall v\in H^r(K).
    \end{equation}
    \item Denote by $\Pi^{BDM}_k$ the Brezzi–Douglas–Marini interpolation onto $[\mathcal{P}_k(K)^3]$. Then for any $ r \in (\frac{1}{2},k+1] $,
    \begin{equation}\label{BDM est}
        \|(I-\Pi^{BDM}_k)v\|_{L^2(K)}\le Ch_K^r|v|_{H^r(K)},\;\forall v\in H^r(K).
    \end{equation}
\end{enumerate}
\end{theorem}
This conclusion is available in \cite[Theorem 5.25]{monk2003finite}, which discussed the RT interpolation, but the proof passes to the BDM one. However, for some unknown reason, the author provided in the theorem statement an estimate using the full norm $ \| v \|_{H^r(K)} $, while its proof only requires the semi norm $ |v|_{H^r(K)} $. \par
The key connection between the $ H (\text{curl}) $ families and the $ H (\text{div}) $ families is the commutative properties. The following conclusion is available in \cite[Proposition 2]{nedelec1986new}, etc.
\begin{theorem}
    $ \forall \bm{u} \in H(\text{curl};K) $ such that $ \nabla \times \bm{u} \in [H^{t_1}(K)]^3 $, $ 1 \le t_1 $, 
    $$ \Pi^{BMD}_{k} (\nabla \times \bm{u}) = \nabla \times (\Pi^{\text{curl; II}}_{k}\bm{u}). $$
\end{theorem}

\section{Convergence under No Extra Assumptions}

\subsection{Auxiliary Results}
For the purpose of analyzing this method, we present some auxiliary conclusions, which we find inconvenient to cite directly from the literature. Therefore, I tailored them to fit our framework, and this subsection is a list of proofs.\par
First, we impose discrete inverse inequalities of fractional ordered Sobolev semi-norms of polynomials. We recall that $ (K,P,\Sigma) $ is an affine generated finite element on $ K \in \mathcal{T}_h $ if there exists a reference element $ (\hat{K},\hat{P},\hat{\Sigma})$ together with an affine mapping $T_K: \hat{K} \to K$ in the form $ T_k(\hat{x}) = B_K \hat{x} + b_K $, where $B_K$ is a $ d\times d $ non-singular matrix.
\begin{lemma}\label{inv ineq}
     Assume that $\{\mathcal{T}_h\}$ is an arbitrary shape regular mesh sequence and $ (K,P,\Sigma) $ is an arbitrary affine generated finite element on $ K\in\mathcal{T}_h $. Assume in addition that $\hat{P}\subset\mathcal{P}_k(\hat{K)}$. Then $\forall p,q\in (1,\infty)$, $\forall m\le [s]$ where $[s]$ means the integer part, $\exists C>0$ such that 
    \begin{equation*}
        |v|_{W^{s,p}(K)}\le C h_K^{m-s-d(\frac{1}{q}-\frac{1}{p})}|v|_{W^{m,q}(K)},\quad\forall v\in P.
    \end{equation*}
\end{lemma}
\begin{proof}
    The case $s\in\mathbb{N}$ is the widely known discrete inverse inequality. Now we only need to illustrate that $\forall s\in (0,1),l\le k$,
    \begin{equation*}
        h_K|w|_{W^{s,p}(K)}\le C h_K^{-s-\frac{d}{q}+\frac{d}{p}}\|w\|_{L^q(K)},\quad\forall w\in \mathcal{P}_l(K).
    \end{equation*}
    First, we notice that both $\|\hat{w}\|_{L^q(\hat{K})}$ and $\|\hat{w}\|_{L^q(\hat{K})}+|\hat{w}|_{W^{s,p}(\hat{K})}$ are norms on $\mathcal{P}_l(\hat{K})$, which is finite dimensional. Therefore,  there exists a constant $C$ that we can assume to be larger than 1, such that for any $\hat{w}\in\mathcal{P}_l(\hat{K})$,
    \begin{equation*}
        \|\hat{w}\|_{L^q(\hat{K})}+|\hat{w}|_{W^{s,p}(\hat{K})}\le C\|\hat{w}\|_{L^q(\hat{K})},
    \end{equation*}
    which means that $|\hat{w}|_{W^{s,p}(\hat{K})}\le (C-1)\|\hat{w}\|_{L^q(\hat{K})}$.
    Next, \cite[Lemma 2.2]{ern2017finite} suggests the following: Consider a transformation $\psi_K(w)=A_k(w\circ T_K)$, and for any $s$ non-integer,
    \begin{align*}
        &|\psi_K(w)|_{W^{s,p}(\hat{K})}\le C\|A_K\|\|B_K\|^{s+\frac{d}{p}}|\det(B_K)|^{-\frac{2}{p}}|w|_{W^{s,p}(K)};\\
        &|w|_{W^{s,p}(K)}\le C\|A_K^{-1}\|\|B_K^{-1}\|^{s+\frac{d}{p}}|\det(B_K)|^\frac{2}{p}|\psi_K(w)|_{W^{s,p}(\hat{K})}.
    \end{align*}
    Take $A_K=I$, and by the shape regular condition
    \begin{align*}
        |w|_{W^{s,p}(K)}&\le Ch_K^{-s-\frac{d}{p}}|\det(B_K)|^\frac{2}{p}|\hat{w}|_{W^{s,p}(\hat{K})}\\
        &\le Ch_K^{-s-\frac{d}{p}}|\det(B_K)|^\frac{2}{p}\|\hat{w}\|_{L^q(\hat{K})}\\
        &\le Ch_K^{-s-\frac{d}{p}}|\det(B_K)|^{\frac{2}{p}-\frac{2}{q}}\|w\|_{L^q(K)}\\
        &\le Ch_K^{-s+\frac{d}{p}-\frac{d}{q}}\|w\|_{L^q(K)}.
    \end{align*}
\end{proof}
Then we consider the following analogue of the conventional discrete trace theorem.
\begin{lemma}\label{trace ineq}
    Assume that $\{\mathcal{T}_h\}$ is an arbitrary shape regular mesh sequence consisting of tetrahedra, $ K$ is a cell of $ \mathcal{T}_h $, and $ \Gamma \subset K $ is a section of plane inside $K$. Then 
    \begin{enumerate}
        \item $ \forall v \in H^1(K) $, $ \| v \|_{L^2(\Gamma)} \le C ( h_K^{-\frac{1}{2}} \| v \|_{L^2(K)} + h_K^{-\frac{1}{2}} \| Dv \|_{L^2(K)}) $; 
        \item $ \forall v_h \in \mathcal{P}_k(K) $, $ \| v_h \|_{L^2(\Gamma)} \le C h_K^{-\frac{1}{2}} \| v_h \|_{L^2(K)} $.
    \end{enumerate}
\end{lemma}
\begin{proof}
    We start by considering the estimate on the reference element. Let $ \hat{K} $ be the standard tetrahedron in $ \mathbb{R}^3 $, i.e, the four vertices are $ (0,0,0) $, $ (1,0,0) $, $ (0,1,0) $, and $ (0,0,1) $. Notice that for any section of a plane in $ \hat{K} $, the four angles it forms with the four faces of $\hat{K}$ can neither be all orthogonal nor be all of unlimited large slope, simultaneously. Thus, we can assume without loss of generality that any plane section $ \hat{\Gamma} $ shall be of the form $ \hat{\Gamma} = \hat{K} \cap P $, where
    $$ P = \{(x_1, x_2, x_3) | x_3 = \zeta(x_1,x_2) \}, $$
    and $ \sqrt{1 + |\partial_1 \zeta|^2 + |\partial_2 \zeta|^2 } < C $ for some $ C > 0 $ independent of the configuration of $P$. \par
    Therefore, for any $ \hat{v} \in H^1(\hat{K}) $, since 
    $$ \hat{v}(x_1, x_2, x_3) = \hat{v}(x_1, x_2, 0) + \int_0^{x_3} \partial_3 \hat{v}(x_1, x_2, d) ds, $$
    we denote by $\hat{F}_4$ the face $ \{(x_1, x_2, x_3) \in \hat{K} | x_3 = 0 \} $, and 
    \begin{align*}
        \int_{\hat{\Gamma}} |\hat{v}|^2 ds & \le \int_{\hat{F}_4} |\hat{v}(x_1, x_2, \zeta(x_1, x_2))|^2 \sqrt{1 + |\partial_1 \zeta|^2 + |\partial_2 \zeta|^2 } dx_1 dx_2 \\
        & \le C ( \int_{\hat{F}_4} |\hat{v}(x_1, x_2, 0)|^2 dx_1 dx_2 + \int_{\hat{F}_4} \int_0^{\zeta(x_1, x_2)} | \hat{D} \hat{v}(x_1, x_2, s)|^2 ds dx_1 dx_2) \\
        & \le C \| \hat{v} \|_{L^2(\hat{F}_4)}^2 + \| \hat{D} \hat{v} \|_{L^2(\hat{K})}^2 \\
        & \le C \| \hat{v} \|_{L^2(\hat{K})}^2 + \| \hat{D} \hat{v} \|_{L^2(\hat{K})}^2.
    \end{align*}
    We proceed by applying the affine mapping $ T_K :\hat{K} \to K $. Note that any concerned plane section $ \Gamma \subset K $ will be pulled back to some $ \hat{\Gamma} \subset \hat{K} $, and for any $ v \in H^1_0(K) $,
    $$ h_K^{-2} \| v \|_{L^2(\Gamma)}^2 \le C \| v \circ T_K \|_{L^2(\hat{\Gamma})}, $$
    while 
    $$ h_K^{-3} \| v \|_{L^2(K)}^2 + h_K^{-1} \| D v \|_{L^2(K)}^2 \le C \| v \circ T_K \|_{L^2(\hat{K})}^2 + \| \hat{D} v \circ T_K \|_{L^2(\hat{K})}^2. $$
    Thus, we have deduced 
    $$ \| v \|_{L^2(\Gamma)} \le C ( h_K^{-\frac{1}{2}} \| v \|_{L^2(K)} + h_K^{-\frac{1}{2}} \| Dv \|_{L^2(K)}). $$
    The inequality for polynomial space derives from the discrete inverse inequality \ref{inv ineq}.
\end{proof}

\subsection{Proof of Convergence by the Compactness Argument}
Now we prove the unconditional convergence result as promised.
\begin{theorem}
    The numerical solutions to \ref{ip mtd}, $ \{(\bm{u}_h,p_h)\}_{ h \in \mathcal{H} } $, converge to the weak solution to \ref{weak form}, $ (\bm{u},p) $, strongly in $ H_0(\text{curl}; \Omega) \times H^1_0(\Omega) $, i.e.,
    \begin{align*}
        \lim_{h \to 0}\| \bm{u}_h - \bm{u} \|_{H(\text{curl};\Omega)} = 0 ; \quad \lim_{h \to 0}\| p_h - p \|_{L^2(\Omega)} = 0.
    \end{align*}
\end{theorem}
The strategy is as established in \cite{qiu2023enriched}. Notice that if we can prove the following two lemmas, then this theorem follows immediately by a simple argument of contradiction.
\begin{lemma}\label{step 1}
    The set $ \{(\bm{u}_h,p_h)\}_{ h \in \mathcal{H} } $ or any infinite subset of it, admits a convergent subsequence, which we still denote as $ \{(\bm{u}_h,p_h)\}_{ h } $, that converges to the limit point $ (\bm{u}_0,p_0) \in H_0(\text{curl}^2;\Omega) \times H^1_0(\Omega) $ in the following way:
    \begin{align*}
        \lim_{h \to 0}\| \bm{u}_h - \bm{u}_0 \|_{H(\text{curl};\Omega)} = 0 ; \quad \lim_{h \to 0}\| p_h - p_0 \|_{L^2(\Omega)} = 0.
    \end{align*}
\end{lemma}
\begin{lemma}\label{step 2}
    Any limit point of $ \{(\bm{u}_h,p_h)\}_{ h \in \mathcal{H} } $ raised as in lemma \ref{step 1} coincides with the weak solution of \ref{weak form}, $ (\bm{u},p) $.
\end{lemma}
To prove Lemma \ref{step 1}, we need to invoke the following results from literature. \par
First, we consider the \textit{Hodge mapping} introduced in \cite[(4.8)]{hiptmair2002finite}: In the context, letting $ \varepsilon = I_{3 \times 3} $, the simple connectivity of $ \Omega $ identifies the space $ Z_0(\varepsilon;\Omega) $ in the author's notation with $ H_0(\text{curl}^2;\Omega)\cap H(\text{div }0;\Omega) $. Hence, 
\begin{lemma}
    For simply connected $ \Omega $ and $ \varepsilon = I_{3 \times 3} $, the following operator is well defined:
    $$ \mathtt{H}_\varepsilon: H_0(\text{curl};\Omega) \to H_0(\text{curl};\Omega) \cap H(\text{div }0;\Omega) $$
    such that 
    \begin{equation}\label{Hodge 1}
        \nabla \times \mathtt{H}_\varepsilon (\bm{v}) = \nabla \times \bm{v}; \quad \nabla \cdot \mathtt{H}\varepsilon (\bm{v}) = 0; \quad \forall \bm{v} \in H_0(\text{curl};\Omega).
    \end{equation}
    Moreover, for any $ \bm{v_h} \in V_h $, i.e., the finite element function of zero discrete divergence, \cite[Lemma 4.5]{hiptmair2002finite} implies 
    \begin{equation}\label{Hodge 2}
        \| \bm{v_h} - \mathtt{H}\varepsilon (\bm{v_h}) \|_{L^2(\Omega)} \le C h^{\frac{1}{2}+\delta} \| \nabla \times \bm{v}_h \|_{L^2(\Omega)}
    \end{equation}
\end{lemma}
These results are also available in \cite[(7.14) and Lemma 7.6]{monk2003finite}. Specifically, \cite{hiptmair2002finite} proved that the relationship \ref{Hodge 2} holds for the Whitney forms $\mathcal{W}^1_{k,0}$ (or denoted as the first kind N{\'e}d{\'e}lec family $ H_0(\text{curl};\Omega) \cap N_{k-1}(\mathcal{T}_h) $ as in \cite{monk2003finite}). However, the proof also applies for the second kind N{\'e}d{\'e}lec family, which is the whole polynomial space $ H_0(\text{curl};\Omega) \cap [\mathcal{P}_k(\mathcal{T}_h)]^3 $. \par
Besides, We further invoke \cite[Theorem 2.2]{weber1980local} or \cite[Corollary 3.49]{monk2003finite}
\begin{lemma}
    Under the current setting that $ \Omega $ is simply connected, the following embedding is compact:
    $$ H_0(\text{curl};\Omega) \cap H(\text{div }0;\Omega) \overset{compact}{\hookrightarrow} [L^2(\Omega)]^3. $$
\end{lemma}
Moreover, \cite[(6.7)]{ern2017finite} has raised an \textit{averaging operator} such that
\begin{lemma}
    There exists a well defined operator
    $$ \mathcal{J}^{av}_{h,0}: \mathcal{P}_k(\mathcal{T}_h) \to \mathcal{P}_k(\mathcal{T}_h) \cap H^1_0(\Omega), $$
    such that (\cite[Lemma 6.2]{ern2017finite})
    \begin{equation*}
        | v_h - \mathcal{J}^{av}_{h,0}(v_h) |_{W^{m,p}(K)} \le C h_K^{d(\frac{1}{p}-\frac{1}{r})+\frac{1}{r}-m}\sum_{ F \in \partial K } \| [\![ \nabla \times v_h ]\!] \|_{L^2(F)},\; \forall v_h \in \mathcal{P}_k(K), \; m \le k+1.
    \end{equation*}
\end{lemma}
\begin{proof}[Proof of Lemma \ref{step 1}]
    We denote by $ \overset{X}{\to} $ and $ \overset{X}{\rightharpoonup} $ the strong convergence and the weak convergence in the Banach space $ X $, respectively. 
    \begin{enumerate}
        \item We start by considering the \textit{Hodge mapping} of $ \{ \bm{u}_h\} $, 
        $$ \mathtt{H}_\varepsilon (\bm{u}_h) \in H_0(\text{curl}^2;\Omega) \cap H(\text{div }0;\Omega) $$
        such that by \ref{Hodge 1} and \ref{Hodge 2}, 
        \begin{equation*}
            \begin{aligned}
                \nabla \times \mathtt{H}_\varepsilon (\bm{u}_h) &= \nabla \times \bm{v}; \\
                \quad \nabla \cdot \mathtt{H}_\varepsilon (\bm{u}_h) & = 0; \\
                \| \bm{u_h} - \mathtt{H}_\varepsilon (\bm{u_h}) \|_{L^2(\Omega)} & \le C h^{\frac{1}{2}+\delta} \| \nabla \times \bm{u}_h \|_{L^2(\Omega)}.
            \end{aligned}
        \end{equation*}
        We emphasize that the constant $ C $ raised in \ref{well-posed} is uniform for the mesh size $ h $, which means that $ \{ \bm{u}_h\} $ is bounded in $ H_0(\text{curl};\Omega) $. Therefore, $ \{\mathtt{H}_\varepsilon (\bm{u}_h) \} $ is also bounded in $ H_0(\text{curl};\Omega) \cap H(\text{div }0;\Omega) $. Thus, the \textit{compactness} of the following embedding
        $$ H_0(\text{curl};\Omega) \cap H(\text{div }0;\Omega) \overset{c}{\hookrightarrow} [L^2(\Omega)]^3 $$
        assigns to $ \{\mathtt{H}_\varepsilon (\bm{u}_h) \} $ an $ [L^2(\Omega)]^3 $ strongly limit point, which we call $ \bm{u}_0 $, and by \ref{Hodge 2}, $ \bm{u}_0 $ is also an $ [L^2(\Omega)]^3 $ strongly limit point of $ \{ \bm{u}_h\} $. 
        \item Notice that the condition $ \bm{u}_h \in H_0(\text{curl};\Omega) $ enforces that $  \bm{n}_F \times [\![ \bm{u}_h ]\!] = 0 $ on $ \mathcal{F}_h $, which implies $  \bm{n}_F \cdot [\![ \nabla \times \bm{u}_h ]\!] = 0 $ on $ \mathcal{F}_h $. That is to say, the jump of $ \nabla \times \bm{u}_h $ only possesses tangential components:
        $$ \| [\![ \nabla \times \bm{u}_h ]\!] \|_{L^2(F)} = \|  \bm{n}_F \times [\![ \nabla \times \bm{u}_h ]\!] \|_{L^2(F)}, \forall F \in \mathcal{F}_h. $$
        Therefore, considering the average mapping $ \{ \mathcal{J}^{av}_{h,0} (\nabla \times \bm{u}_h) \} $, we have
        \begin{equation}\label{average operator}
            \begin{aligned}
                \sum_{K \in \mathcal{T}_h} \| \nabla \times \bm{u}_h - \mathcal{J}^{av}_{h,0}(\nabla \times \bm{u}_h)) \|_{L^2(K)} & \le C \sum_{ F \in \mathcal{F}_h } h_F^{\frac{1}{2}} \|  \bm{n}_F \times [\![ \nabla \times \bm{u}_h ]\!] \|_{L^2(F)}; \\
                \sum_{K \in \mathcal{T}_h} \| D(\nabla \times \bm{u}_h) - D\mathcal{J}^{av}_{h,0}(\nabla \times \bm{u}_h)) \|_{L^2(K)} & \le C \sum_{ F \in \mathcal{F}_h } h_F^{-\frac{1}{2}} \|  \bm{n}_F \times [\![ \nabla \times \bm{u}_h ]\!] \|_{L^2(F)}. 
            \end{aligned}
        \end{equation}
        Notice that $ \sum_{ F \in \mathcal{F}_h } h_F^{-\frac{1}{2}} \|  \bm{n}_F \times [\![ \nabla \times \bm{u}_h ]\!] \|_{L^2(F)} $ is bounded by $ \|\bm{f}\|_{L^2(\Omega)} $, as in \ref{energy estimate}. Thus, $ \{ \mathcal{J}^{av}_{h,0} (\nabla \times \bm{u}_h) \} $ turns out to be a bounded sequence in $ [H^1_0(\Omega)]^3 $, and hence admits an $ H^1_0(\Omega) $-weak limit $ \bm{s}_0 $. Moreover, the Rellich–Kondrachov Theorem suggests that $ \bm{s}_0 $ is also an $ [L^2(\Omega)]^3 $ strongly limit of $ \{ \mathcal{J}^{av}_{h,0} (\nabla \times \bm{u}_h) \} $, and by the first line of \ref{average operator}, an $ [L^2(\Omega)]^3 $ strongly limit of $ \{ \nabla \times \bm{u}_h \} $. \par 
        Now we figure out that $ \bm{s}_0 = \nabla \times \bm{u}_0 $. This follows from the argument
        $$ (\bm{u}_0, \bm{\phi})_\Omega = \lim_{h \to \infty} (\nabla \times \bm{u}_h, \bm{\phi} )_\Omega = \lim_{h \to \infty} - (\bm{u}_h, \nabla \times \bm{\phi} )_\Omega = - (\bm{u}_0, \nabla \times \bm{\phi} )_\Omega, \; \forall \bm{\phi} \in [C^\infty_0(\Omega)]^3. $$
        \item By the energy estimate $ \| p_h \|_{ H^1_0(\Omega) } \le C \| \bm{f} \|_{L^2(\Omega)} $, $ \{ p_h \} $ admits an $ H^1_0(\Omega) $ weakly, $ L^2(\Omega) $ strongly limit $p_0$. 
        \end{enumerate}
    In conclusion, we have raised $ \bm{u}_0 \in H_0 (\text{curl}^2;\Omega) $ and $ p_0 \in H^1_0(\Omega) $ such that a subsequence of $ \{(\bm{u}_h,p_h)\}_{ h \in \mathcal{H} } $, which we still denote by the same notation, satisfies
    \begin{equation}\label{step 1.1}
        \begin{aligned}
            \bm{u}_h & \overset{L^2(\Omega)}{\to} \bm{u}_0, \quad h \to 0 ;\\
            \nabla \times \bm{u}_h & \overset{L^2(\Omega)}{\to} \nabla \times \bm{u}_0, \quad h \to 0 ;\\
            \mathcal{J}^{av}_{h,0} (\nabla \times \bm{u}_h) & \overset{H^1_0(\Omega)}{\rightharpoonup} \nabla \times \bm{u}_0, \quad h \to 0 ;\\
            \nabla \cdot \bm{u}_0 & = 0;\\
            p_h & \overset{H^1_0(\Omega)}{\rightharpoonup} p_0, \quad h \to 0. 
        \end{aligned}
    \end{equation}
\end{proof}

\begin{proof}[proof of lemma \ref{step 2}]
    Next, we show that $ (\bm{u}_0,p_0) $ raised above actually solves the weak formulation \ref{weak form}. By \cite[Theorem 2.3]{chen2021analysis}, which states that $ [C^\infty_0(\Omega)]^3 $ is dense in $ H_0(\text{curl}^2; \Omega) $, we just need to verify that \ref{weak form} holds per test function from $ [C^\infty_0(\Omega)]^3 $. For an arbitrary $ \bm{v} \in [C^\infty_0(\Omega)]^3 $, we set $ \bm{v}_h := \Pi^{\text{curl}}_{h,k}\bm{v} $ as the second N{\'e}d{\'e}lec interpolation. Because of the smoothness of $ \bm{v} $ and the assumption that $ k \ge 2 $, we can utilize \ref{Nedelec est 1} and \ref{Nedelec est 2} to deduce the approximation results at least as good as below:
    \begin{equation}\label{est on test v}
        \begin{aligned}
            \| \bm{v} - \bm{v}_h \|_{L^2(K)} & \le C h_K^2 ( \| \bm{v} \|_{H^2(K)} + h_K \| \nabla \times \bm{v} \|_{H^2(K)}) ; \\
            \| \nabla \times (\bm{v} - \bm{v}_h) \|_{L^2(K)} & \le C h_K^2 \| \nabla \times \bm{v} \|_{H^2(K)}; \\
            \| (\nabla \times)^2 (\bm{v} - \bm{v}_h) \|_{L^2(K)} & \le C h_K^{1} \| \nabla \times \bm{v} \|_{H^{2}(K)}; \\
            h_F^{-\frac{1}{2}} \|  \bm{n}_F \times [\![\nabla\times (\bm{v} - \bm{v}_h) ]\!]  \|_{L^2(F)} & \le C h_K^{1} \| \nabla \times \bm{v} \|_{H^{2}(K)}; \\
            h_F^{\frac{1}{2}} \| \{\!\{A(\nabla\times\nabla\times (\bm{v} - \bm{v}_h)\}\!\} \|_{L^2(F)} & \le C ( h_K^{1} \| \nabla \times \bm{v} \|_{H^{2}(K)} + h_K^{1} \| (\nabla \times)^2 \bm{v} \|_{H^{1}(K)}  ).
        \end{aligned}
    \end{equation}
    Now we split the weak formulation as follows:
    \begin{align*}
        & (A (\nabla \times)^2 \bm{u}_0, (\nabla \times)^2 \bm{v} )_\Omega + ( \nabla p_0, \bm{v} )_{ \Omega } \\
        & = \sum_{ K \in \mathcal{T}_h }(A (\nabla \times)^2 (\bm{u}_0 - \bm{u}_h), (\nabla \times)^2 \bm{v} )_{ K } + ( \nabla (p_0 - p_h), \bm{v} )_{ K } + \sum_{F\in\mathcal{F}_h} \langle \{\!\{ A(\nabla\times)^2 \bm{v}) \}\!\},  \bm{n}_F \times [\![ \nabla\times \bm{u}_h ]\!] \rangle_{F} \\
        & \quad + \sum_{ K \in \mathcal{T}_h } (A (\nabla \times)^2 \bm{u}_h, (\nabla \times)^2 (\bm{v} - \bm{v}_h ) )_{ K } + ( \nabla p_h, \bm{v} - \bm{v}_h )_{ K } \\
        & \quad + \sum_{ K \in \mathcal{T}_h } (A (\nabla \times)^2 \bm{u}_h, (\nabla \times)^2 \bm{v}_h ) )_{ K } + ( \nabla p_h, \bm{v}_h )_{ K } -  \sum_{F\in\mathcal{F}_h} \langle \{\!\{ A(\nabla\times)^2 \bm{v}) \}\!\},  \bm{n}_F \times [\![ \nabla\times \bm{u}_h ]\!] \rangle_{F} \\
        & =: I + I\!I + I\!I\!I
    \end{align*}
    \begin{enumerate}
        \item To analyze $ I $, we first decompose as
        $$ \sum_{ K \in \mathcal{T}_h }(A (\nabla \times)^2 (\bm{u}_0 - \bm{u}_h), (\nabla \times)^2 \bm{v} )_{ K } = \sum_{ K \in \mathcal{T}_h } \sum_{ i \in \Lambda } (A (\nabla \times)^2 (\bm{u}_0 - \bm{u}_h), (\nabla \times)^2 \bm{v} )_{ K \cap \Omega_i }. $$ 
        We claim that the integration by parts on each $ K \cap \Omega_i $ are admissible. To see this, note that the polyhedral $ \Omega_i $ admits a partition into simplex (tetrahedra): $ \Omega_i : = \bigsqcup_{j \in J} P_j $, and each region $ P_j \cap K $ is convex, thus assigned with a well defined outward unit normal vector field $ \bm{n}_j $. Nevertheless, the will be no conflict on the direction of $ \bm{n}_j $ for a face shared by two simplex, since $ \Omega_i $ is Lipschitz. In conclusion, the computation goes 
        \begin{align*}
            & (A (\nabla \times)^2 (\bm{u}_0 - \bm{u}_h), (\nabla \times)^2 \bm{v} )_{ K \cap \Omega_i } \\
            & = (A \nabla \times (\nabla \times \bm{u}_0 - \mathcal{J}^{av}_{h0}(\nabla \times \bm{u}_h)), (\nabla \times)^2 \bm{v} )_{ K \cap \Omega_i } \\
            & \quad + (A \nabla \times (\mathcal{J}^{av}_{h0}(\nabla \times \bm{u}_h) - \nabla \times \bm{u}_h ), (\nabla \times)^2 \bm{v} )_{ K \cap \Omega_i } \\
            & = (A \nabla \times (\nabla \times \bm{u}_0 - \mathcal{J}^{av}_{h0}(\nabla \times \bm{u}_h)), (\nabla \times)^2 \bm{v} )_{ K \cap \Omega_i } \\
            & \quad + (\mathcal{J}^{av}_{h0}(\nabla \times \bm{u}_h) - \nabla \times \bm{u}_h, \nabla \times A (\nabla \times)^2 \bm{v} )_{ K \cap \Omega_i } \\
            & \quad + \langle  \bm{n}_{ \partial (\text{int}(K) \cap \Omega_i) } \times ( \mathcal{J}^{av}_{h0}(\nabla \times \bm{u}_h) - \nabla \times \bm{u}_h ), A (\nabla \times)^2 \bm{v} \rangle_{ \partial (\text{int} (K) \cap \Omega_i) }.
        \end{align*}
        This expression seems to derive a gigantic integral on $ \bigcup_{ K \in \mathcal{T}_h } \bigcup_{ i \in \Lambda } \partial (\text{int} (K) \cap \Omega_i) $ when summing up. We shall figure out, however, that most terms cancel each other indeed. To reveal this, we denote
        $$ \Sigma := \bigcup_{ i \in \Lambda } (\partial \Omega_i) \backslash \partial \Omega, $$
        and a unit normal vector field $  \bm{n}_\Sigma $ consistent with the orientation is assigned. Notice that $ \partial (\text{int} (K) \cap \Omega_i) \subset \partial K \cup \partial \Omega_i $, where $ \partial \Omega_i $ is contained in the interior of $ K $ only if it belongs to $ \Sigma $. Therefore, we can always apply the following splitting: 
        $$ \bigcup_{ i \in \Lambda } \partial (\text{int} (K) \cap \Omega_i) = \partial K \sqcup (\Sigma \cap \text{int} (K)), $$
        where the interior of $ K $ is defined to be open, disjoint with $ \partial K $. Nevertheless, the direction of $  \bm{n}_{ \partial (\text{int} (K) \cap \Omega_i) } $ is the same as $  \bm{n}_{ \partial K } $ on $ \partial K $. Thus,
        \begin{align*}
            & \sum_{ K \in \mathcal{T}_h } \sum_{ i \in \Lambda } \langle  \bm{n}_{ \partial (\text{int} (K) \cap \Omega_i) } \times ( \mathcal{J}^{av}_{h0}(\nabla \times \bm{u}_h) - \nabla \times \bm{u}_h ), A (\nabla \times)^2 \bm{v} \rangle_{ \partial (\text{int} (K) \cap \Omega_i) } \\
            & = \sum_{ K \in \mathcal{T}_h } \langle  \bm{n}_{ \partial K } \times ( \mathcal{J}^{av}_{h0}(\nabla \times \bm{u}_h) - \nabla \times \bm{u}_h ), A (\nabla \times)^2 \bm{v} \rangle_{ \partial K } \\
            & \quad + \sum_{ K \in \mathcal{T}_h } \sum_{ i \in \Lambda } \langle  \bm{n}_{ \partial (\text{int} (K) \cap \Omega_i) } \times ( \mathcal{J}^{av}_{h0}(\nabla \times \bm{u}_h) - \nabla \times \bm{u}_h ), A (\nabla \times)^2 \bm{v} \rangle_{ \partial \Omega_i \cap \text{int} (K)}.
        \end{align*}
        \begin{itemize}
            \item Because of the assertion $ \bm{v} \in C^\infty_c(\Omega) $, the integral on $ \partial \Omega $ vanishes:
            \begin{align*}
                \sum_{ K \in \mathcal{T}_h } \langle  \bm{n}_{ \partial K } \times ( \mathcal{J}^{av}_{h0}(\nabla \times \bm{u}_h) - \nabla \times \bm{u}_h ), A (\nabla \times)^2 \bm{v} \rangle_{ \partial K \cap \partial \Omega } = 0.
            \end{align*}
            It costs nothing to rewrite this as 
            $$ \sum_{ F \in \mathcal{T}^\partial_h } \langle  \bm{n}_{ F } \times [\![ \mathcal{J}^{av}_{h0}(\nabla \times \bm{u}_h) - \nabla \times \bm{u}_h ]\!], \{\!\{ A (\nabla \times)^2 \bm{v} \}\!\} \rangle_{ F } $$
            Meanwhile, referring to the definition of $ [\![ \cdot ]\!] $, the calculations on internal interfaces can be coupled as
            \begin{align*}
                & \sum_{ K \in \mathcal{T}_h } \langle  \bm{n}_{ \partial K } \times ( \mathcal{J}^{av}_{h0}(\nabla \times \bm{u}_h) - \nabla \times \bm{u}_h ), A (\nabla \times)^2 \bm{v} \rangle_{ \partial K \cap \text{int} (\Omega) } \\
                & = \sum_{ F \in \mathcal{T}^\text{int}_h } \langle  \bm{n}_{ \partial K^+ } \times ( \mathcal{J}^{av}_{h0}(\nabla \times \bm{u}_h) - \nabla \times \bm{u}_h )|_{K^+}, A (\nabla \times)^2 \bm{v}|_{K^+} \rangle_{ F } \\ 
                & \quad + \langle  \bm{n}_{ \partial K^- } \times ( \mathcal{J}^{av}_{h0}(\nabla \times \bm{u}_h) - \nabla \times \bm{u}_h )|_{K^-}, A (\nabla \times)^2 \bm{v}|_{K^-} \rangle_{ F } \\
                & = \sum_{ F \in \mathcal{T}^\text{int}_h } \langle  \bm{n}_{ F } \times [\![ \mathcal{J}^{av}_{h0}(\nabla \times \bm{u}_h) - \nabla \times \bm{u}_h ]\!], \{\!\{ A (\nabla \times)^2 \bm{v} \}\!\} \rangle_{ F } \\
                & \quad + \langle  \bm{n}_{ F } \times \{\!\{ \mathcal{J}^{av}_{h0}(\nabla \times \bm{u}_h) - \nabla \times \bm{u}_h \}\!\}, [\![ A (\nabla \times)^2 \bm{v} ]\!] \rangle_{ F }.
            \end{align*}
            Notice that $ [\![ A (\nabla \times)^2 \bm{v} ]\!] $ is non-vanishing on $ \Sigma $. Therefore, we assemble
            \begin{align*}
                & \sum_{ K \in \mathcal{T}_h } \langle  \bm{n}_{ \partial K } \times ( \mathcal{J}^{av}_{h0}(\nabla \times \bm{u}_h) - \nabla \times \bm{u}_h ), A (\nabla \times)^2 \bm{v} \rangle_{ \partial K } \\
                & = \sum_{F \in \mathcal{T}_h } \langle  \bm{n}_{ F } \times [\![ \mathcal{J}^{av}_{h0}(\nabla \times \bm{u}_h) - \nabla \times \bm{u}_h ]\!], \{\!\{ A (\nabla \times)^2 \bm{v} \}\!\} \rangle_{ F } \\
                & \quad + \langle  \bm{n}_{ F } \times \{\!\{ \mathcal{J}^{av}_{h0}(\nabla \times \bm{u}_h) - \nabla \times \bm{u}_h \}\!\}, [\![ A (\nabla \times)^2 \bm{v} ]\!] \rangle_{ F \cap \Sigma }.
            \end{align*}
            \item Notice that 
            $$ \bigcup_{ i \in \Lambda } \partial (\text{int} (K) \cap \Omega_i) \cap \text{int} (K) = \Sigma \cap \text{int} (K), $$ 
            where $ \Sigma \cap \text{int} (K) = \bigcup_{i \ne j} \partial \Omega_i \cap \partial \Omega_j \cap \text{int} (K) $, with the unit normal vectors $  \bm{n}_{ \partial (\text{int} (K) \cap \Omega_i) } =  \bm{n}_\Sigma $, $  \bm{n}_{ \partial (K \cap \Omega_j) } = -  \bm{n}_\Sigma $. Therefore, 
            \begin{align*}
                & \sum_{ K \in \mathcal{T}_h } \sum_{ i \in \Lambda } \langle  \bm{n}_{ \partial (\text{int} (K) \cap \Omega_i) } \times ( \mathcal{J}^{av}_{h0}(\nabla \times \bm{u}_h) - \nabla \times \bm{u}_h ), A (\nabla \times)^2 \bm{v} \rangle_{ \partial \Omega_i \cap \text{int} (K)} \\
                & = \sum_{ K \in \mathcal{T}_h } \langle  \bm{n}_{ F } \times \mathcal{J}^{av}_{h0}(\nabla \times \bm{u}_h) - \nabla \times \bm{u}_h, [\![ A (\nabla \times)^2 \bm{v} ]\!] \rangle_{ \Sigma \cap \text{int} (K)}
            \end{align*}
            Since $ \nabla \times \bm{u}_h $ and $ \mathcal{J}^{av}_{h0}(\nabla \times \bm{u}_h) $ are smooth within the interior of $ K $, we can rewrite the term above as 
            $$ \sum_{ K \in \mathcal{T}_h } \langle  \bm{n}_{ F } \times \{\!\{ \mathcal{J}^{av}_{h0}(\nabla \times \bm{u}_h) - \nabla \times \bm{u}_h \}\!\}, [\![ A (\nabla \times)^2 \bm{v} ]\!] \rangle_{ \Sigma \cap \text{int} (K)}. $$
        \end{itemize}
        In conclusion, the terms 
        $$ \langle  \bm{n}_{ F } \times \{\!\{ \mathcal{J}^{av}_{h0}(\nabla \times \bm{u}_h) - \nabla \times \bm{u}_h \}\!\}, [\![ A (\nabla \times)^2 \bm{v} ]\!] \rangle $$
        are calculated on two types of shapes: $ F \cap \Sigma $ for $ F \in \mathcal{F}_h $, and $ \Sigma \cap \text{int} (K) $ for $ K \in \mathcal{T}_h $, exactly covering $ \Sigma $. Hence, we summarize
        \begin{align*}
            & \sum_{ K \in \mathcal{T}_h }(A (\nabla \times)^2 (\bm{u}_0 - \bm{u}_h), (\nabla \times)^2 \bm{v} )_{ K } \\
            & = \sum_{ K \in \mathcal{T}_h }(A \nabla \times (\nabla \times \bm{u}_0 - \mathcal{J}^{av}_{h0}(\nabla \times \bm{u}_h)), (\nabla \times)^2 \bm{v} )_{ K } \\
            & \quad + \sum_{ K \in \mathcal{T}_h } (\mathcal{J}^{av}_{h0}(\nabla \times \bm{u}_h) - \nabla \times \bm{u}_h, \nabla \times A (\nabla \times)^2 \bm{v} )_{ K } \\
            & \quad + \sum_{F \in \mathcal{T}_h } \langle  \bm{n}_{ F } \times [\![ \mathcal{J}^{av}_{h0}(\nabla \times \bm{u}_h) - \nabla \times \bm{u}_h ]\!], \{\!\{ A (\nabla \times)^2 \bm{v} \}\!\} \rangle_{ F } \\
            & \quad + \langle  \bm{n}_{ \Sigma } \times \{\!\{ \mathcal{J}^{av}_{h0}(\nabla \times \bm{u}_h) - \nabla \times \bm{u}_h \}\!\}, [\![ A (\nabla \times)^2 \bm{v} ]\!] \rangle_{ \Sigma },
        \end{align*}
        and 
        \begin{align*}
            I & = \sum_{ K \in \mathcal{T}_h } (A (\nabla \times)^2 (\bm{u}_0 - \bm{u}_h), (\nabla \times)^2 \bm{v} )_{ K } + ( \nabla (p_0 - p_h), \bm{v} )_{ K } \\
            & \quad + \sum_{F\in\mathcal{F}_h} \langle \{\!\{ A(\nabla\times)^2 \bm{v}) \}\!\},  \bm{n}_F \times [\![ \nabla\times \bm{u}_h ]\!] \rangle_{F} \\
            & = \sum_{ K \in \mathcal{T}_h } (A \nabla \times (\nabla \times \bm{u}_0 - \mathcal{J}^{av}_{h0}(\nabla \times \bm{u}_h)), (\nabla \times)^2 \bm{v} )_{ K } \\
            & \quad + \sum_{ K \in \mathcal{T}_h } (\mathcal{J}^{av}_{h0}(\nabla \times \bm{u}_h) - \nabla \times \bm{u}_h, \nabla \times A (\nabla \times)^2 \bm{v} )_{ K } \\
            & \quad + \sum_{F \in \mathcal{T}_h } \langle  \bm{n}_{ F } \times [\![ \mathcal{J}^{av}_{h0}(\nabla \times \bm{u}_h) - \nabla \times \bm{u}_h ]\!], \{\!\{ A (\nabla \times)^2 \bm{v} \}\!\} \rangle_{ F } \\
            & \quad + \langle  \bm{n}_{ \Sigma } \times \{\!\{ \mathcal{J}^{av}_{h0}(\nabla \times \bm{u}_h) - \nabla \times \bm{u}_h \}\!\}, [\![ A (\nabla \times)^2 \bm{v} ]\!] \rangle_{ \Sigma } \\
            & \quad + \sum_{ K \in \mathcal{T}_h } ( \nabla (p_0 - p_h), \bm{v} )_{ K } \\
            & \quad + \sum_{F\in\mathcal{F}_h} \langle \{\!\{ A(\nabla\times)^2 \bm{v}) \}\!\},  \bm{n}_F \times [\![ \nabla\times \bm{u}_h ]\!] \rangle_{F} \\
            & =: I_1 + I_2 + I_3 + I_4 + I_5 + I_6
        \end{align*}
        \begin{itemize}
            \item By the third line of \ref{step 1.1}, the weak convergence $ \mathcal{J}^{av}_{h,0} (\nabla \times \bm{u}_h) \overset{H^1_0(\Omega)}{\rightharpoonup} \nabla \times \bm{u}_0 $ implies
            $$ I_1 = (A \nabla \times (\nabla \times \bm{u}_0 - \mathcal{J}^{av}_{h0}(\nabla \times \bm{u}_h)), (\nabla \times)^2 \bm{v} )_\Omega \to 0, \quad h \to 0; $$
            \item By \ref{average operator} and \ref{energy estimate}
            \begin{align*}
                | I_2 | & \le \sum_{ K \in \mathcal{T}_h } | (\mathcal{J}^{av}_{h0}(\nabla \times \bm{u}_h) - \nabla \times \bm{u}_h, \nabla \times A (\nabla \times)^2 \bm{v} )_{ K } | \\
                & \le \sum_{K \in \mathcal{T}_h} \| (\mathcal{J}^{av}_{h0}(\nabla \times \bm{u}_h) - \nabla \times \bm{u}_h \|_{L^2(K)} \| \nabla \times A (\nabla \times)^2 \bm{v} \|_{L^2(K)} \\
                &  \le C \sum_{ F \in \mathcal{F}_h } h_F^{\frac{1}{2}} \|  \bm{n}_F \times [\![ \nabla \times \bm{u}_h ]\!] \|_{L^2(F)} \| A \|_{L^\infty(\Omega)} \| \bm{v} \|_{H^3(\Omega)} \\
                & \le C h \| \bm{f} \|_{L^2(\Omega)} \| A \|_{L^\infty(\Omega)} \| \bm{v} \|_{H^3(\Omega)} \\
                & \to 0, \quad h \to 0.
            \end{align*}
            \item Since $ \mathcal{J}^{av}_{h0}(\nabla \times \bm{u}_h) \in H^1_0(\Omega) $,
            \begin{align*}
                I_3 & = \sum_{F\in\mathcal{F}_h} \langle  \bm{n}_{ F } \times [\![ \mathcal{J}^{av}_{h0}(\nabla \times \bm{u}_h) - \nabla \times \bm{u}_h  ]\!], \{ \! \{ A (\nabla \times)^2 \bm{v} \} \! \} \rangle_{ F } \\
                & = - \sum_{F\in\mathcal{F}_h} \langle  \bm{n}_{ F } \times [\![ \nabla \times \bm{u}_h ]\!], \{ \! \{ A (\nabla \times)^2 \bm{v} \} \! \} \rangle_{ F } \\
                & = - I_6
            \end{align*}
            \item We rearrange $ I_4 $ as
            \begin{align*}
                I_4 = & \langle  \bm{n}_{ \Sigma } \times \{\!\{ \mathcal{J}^{av}_{h0}(\nabla \times \bm{u}_h) - \nabla \times \bm{u}_h \}\!\}, [\![ A (\nabla \times)^2 \bm{v} ]\!] \rangle_{ \Sigma } \\
                & = \sum_{\Sigma \cap \overline{K} \ne \emptyset} \langle  \bm{n}_{ \Sigma } \times [\![ A (\nabla \times)^2 \bm{v} ]\!], \{\!\{ \mathcal{J}^{av}_{h0}(\nabla \times \bm{u}_h) - \nabla \times \bm{u}_h \}\!\} \rangle_{ \Sigma \cap \overline{K} } \\
                & \le ( \sum_{\Sigma \cap \overline{K} \ne \emptyset} h_K \|  \bm{n}_{ \Sigma} \times [\![ A (\nabla \times)^2 \bm{v} ]\!] \|_{ L^2(\Sigma \cap \overline{K}) }^2 )^\frac{1}{2} \\
                & \quad \cdot ( \sum_{\Sigma \cap \overline{K} \ne \emptyset} h_K^{-1} \| \{\!\{ \mathcal{J}^{av}_{h0}(\nabla \times \bm{u}_h) - \nabla \times \bm{u}_h \}\!\} \|_{ L^2(\Sigma \cap \overline{K}) }^2 )^\frac{1}{2}.
            \end{align*}
            where
            \begin{align*}
                \sum_{\Sigma \cap \overline{K} \ne \emptyset} h_K \|  \bm{n}_{ \Sigma \cap \overline{K} } \times [\![ A (\nabla \times)^2 \bm{v} ]\!] \|_{ L^2(\Sigma \cap \overline{K}) }^2 & \le C h \sum_{\Sigma \cap \overline{K} \ne \emptyset} \| [\![ A (\nabla \times)^2 \bm{v} ]\!] \|_{ L^2(\Sigma \cap \overline{K}) }^2 \\
                & = C h \| [\![ A (\nabla \times)^2 \bm{v} ]\!] \|_{ L^2( \Sigma ) }^2 \\
                & \to 0, \quad h \to 0,
            \end{align*}
            which converges to $0$, while by the revised discrete trace inequality in lemma \ref{trace ineq}, \ref{average operator}, and \ref{energy estimate}
            \begin{align*}
                & \sum_{\Sigma \cap \overline{K} \ne \emptyset} h_K^{-1} \| \{\!\{ \mathcal{J}^{av}_{h0}(\nabla \times \bm{u}_h) - \nabla \times \bm{u}_h \}\!\} \|_{ L^2(\Sigma \cap \overline{K}) }^2 \\
                & \le C \sum_{K\in\mathcal{T}_h} h_K^{-2} \| \mathcal{J}^{av}_{h0}(\nabla \times \bm{u}_h) - \nabla \times \bm{u}_h \|_{ L^2(K) }^2 \\
                & \le C \sum_{ F \in \mathcal{F}_h } h_F^{-1} \|  \bm{n}_F \times [\![ \nabla \times \bm{u}_h ]\!] \|_{L^2(F)}^2 \\
                & \le C \| \bm{f} \|_{L^2(\Omega)},
            \end{align*}
            which is bounded.
            \item By line 5 of \ref{step 1.1} that $ p_h \overset{H^1_0(\Omega)}{\rightharpoonup} p_0 $, we have
            $$ I_5 = ( \nabla (p_0 - p_h), \bm{v} )_\Omega  \to 0, \quad h \to 0. $$
        \end{itemize}
        So far we have shown that
        \begin{align*}
                I_1 \to 0, \quad h \to 0.
        \end{align*}
        \item We combine \ref{energy estimate} and \ref{est on test v} to estimate $ I \! I $:
        \begin{align*}
            I \! I & = \sum_{ K \in \mathcal{T}_h } (A (\nabla \times)^2 \bm{u}_h, (\nabla \times)^2 (\bm{v} - \bm{v}_h ) )_{ K } + ( \nabla p_h, \bm{v} - \bm{v}_h )_{ K } \\
            & \le (\sum_{ K \in \mathcal{T}_h }\| A (\nabla \times)^2 \bm{u}_h \|_{ L^2(K) }^2)^\frac{1}{2} (\sum_{ K \in \mathcal{T}_h }\| (\nabla \times)^2 (\bm{v} - \bm{v}_h ) \|_{L^2(K)}^2)^\frac{1}{2} \\
            & \quad + \| \nabla p_h \|_{L^2(\Omega)} (\sum_{ K \in \mathcal{T}_h }\| \bm{v} - \bm{v}_h \|_{L^2(F)}^2)^\frac{1}{2} \\
            & \le C \| A \|_{L^\infty(\Omega)} \| \bm{f} \|_{L^2(\Omega)} \sum _{K\in\mathcal{T}_h} h_K \| \nabla \times \bm{v} \|_{H^2(\Omega)} \\ 
            & \quad + C \| \bm{f} \|_{L^2(\Omega)} \sum _{K\in\mathcal{T}_h} h_K^2 ( \| \bm{v} \|_{H^2(K)} + h_K \| \nabla \times \bm{v} \|_{H^2(K)}) \\
            & \to 0, \quad h \to 0.
        \end{align*}
        \item Since $ \{(\bm{u}_h,p_h)\} $ is the numerical solution to \ref{ip mtd}, we can rearrange $ I_3 $ as 
        \begin{align*}
            I \! I \! I & = \sum_{ K \in \mathcal{T}_h } (A (\nabla \times)^2 \bm{u}_h, (\nabla \times)^2 \bm{v}_h ) )_{ K } + ( \nabla p_h, \bm{v}_h )_{ K } -  \sum_{F\in\mathcal{F}_h} \langle \{\!\{ A(\nabla\times)^2 \bm{v} \}\!\},  \bm{n}_F \times [\![ \nabla\times \bm{u}_h ]\!] \rangle_{F} \\
            & = \sum_{F\in\mathcal{F}_h} \langle \{\!\{ A ( \nabla \times )^2 \bm{u}_h \}\!\},  \bm{n}_F \times [\![ \nabla\times  \bm{v}_h ]\!] \rangle_{F} \\
            & \quad + \sum_{F \in \mathcal{F}_h }\langle \{\!\{ A ( \nabla \times )^2 (\bm{v}_h -\bm{v}) \}\!\},  \bm{n}_F \times [\![ \nabla \times \bm{u}_h ]\!] \rangle_{F} \\
            & \quad - \sum_{F \in \mathcal{F}_h} \frac{\tau}{h_F} \langle  \bm{n}_F\times[\![ \nabla\times \bm{u}_h ]\!],  \bm{n}_F \times [\![ \nabla\times \bm{v}_h ]\!] \rangle_{F} \\
            & \quad + (\bm{f}, \bm{v}_h )_\Omega \\
            & =: I \! I \! I_1 + I \! I \! I_2 + I \! I \! I_3 + I \! I \! I_4
        \end{align*}
        \begin{itemize}
            \item Notice that $  [\![ \nabla\times \bm{v}_h ]\!] =  [\![ \nabla\times (\bm{v}_h - \bm{v}) ]\!] $ since $ \bm{v} \in [C^\infty_0(\Omega)]^3 $. Thus, applying \ref{est on test v} and the discrete inverse inequality,
            \begin{align*}
                I \! I \! I_1 & = \sum_{F\in\mathcal{F}_h} \langle \{\!\{ A ( \nabla \times )^2 \bm{u}_h \}\!\},  \bm{n}_F \times [\![ \nabla\times  (\bm{v}_h - \bm{v}) ]\!] \rangle_{F} \\
                & \le \sum_{F\in\mathcal{F}_h} h_F^\frac{1}{2} \| \{\!\{ A ( \nabla \times )^2 \bm{u}_h \}\!\} \|_{L^2(F)}  h_F^{-\frac{1}{2}}\|  \bm{n}_F \times [\![ \nabla\times (\bm{v}_h - \bm{v}) ]\!] \|_{L^2(F)} \\
                & \le C \sum_{K\in\mathcal{T}_h} \| A ( \nabla \times )^2 \bm{u}_h \|_{L^2(K)} h_K^{1} \| \nabla \times \bm{v} \|_{H^{2}(K)} \\
                & \to 0, \quad h \to 0.
            \end{align*} 
            \item Apply \ref{est on test v} and \ref{energy estimate},
            \begin{align*}
                I \! I \! I_2 & = \sum_{F \in \mathcal{F}_h }\langle \{\!\{ A ( \nabla \times )^2 (\bm{v}_h -\bm{v}) \}\!\},  \bm{n}_F \times [\![ \nabla \times \bm{u}_h ]\!] \rangle_{F} \\
                & \le \sum_{F\in\mathcal{F}_h} h_F^\frac{1}{2} \| \{\!\{ A ( \nabla \times )^2 (\bm{v}_h -\bm{v}) \}\!\} \|_{L^2(F)}  h_F^{-\frac{1}{2}}\|  \bm{n}_F \times [\![ \nabla\times \bm{u}_h ]\!] \|_{L^2(F)} \\
                & \le C \sum_{K\in\mathcal{T}_h} ( h_K^{1} \| \nabla \times \bm{v} \|_{H^{2}(K)} + h_K^{1} \| (\nabla \times)^2 \bm{v} \|_{H^{1}(K)}  ) \| \bm{f} \|_{L{2}(\Omega)} \\
                & \to 0, \quad h \to 0.
            \end{align*}
            \item Apply \ref{est on test v} and \ref{energy estimate} again,
            \begin{align*}
                I \! I \! I_3 & = \sum_{F \in \mathcal{F}_h} \frac{\tau}{h_F} \langle  \bm{n}_F\times[\![ \nabla\times \bm{u}_h ]\!],  \bm{n}_F \times [\![ \nabla\times \bm{v}_h ]\!] \rangle_{F} \\
                & \le C\sum_{F\in\mathcal{F}_h} h_F^{-\frac{1}{2}}\|  \bm{n}_F \times [\![ \nabla\times \bm{u}_h ]\!] \|_{L^2(F)} h_F^{-\frac{1}{2}}\|  \bm{n}_F \times [\![ \nabla\times (\bm{v}_h - \bm{v}) ]\!] \|_{L^2(F)} \\
                & \le C \| \bm{f} \|_{L{2}(\Omega)} \sum_{K\in\mathcal{T}_h} ( h_K^{1} \| \nabla \times \bm{v} \|_{H^{2}(K)} )  \\
                & \to 0, \quad h \to 0.
            \end{align*}
            \item By \ref{est on test v}, $ I \! I \! I_4 = (\bm{f}, \bm{v}_h )_\Omega \to (\bm{f}, \bm{v})_\Omega $ as $ h \to 0 $. 
            \end{itemize}
        In conclusion, $ I \! I \! I \to (\bm{f}, \bm{v})_\Omega $ as $ h \to 0 $.
        \end{enumerate}
        Now taking the limit as above, we have deduced that
        $$ (A (\nabla \times)^2 \bm{u}_0, (\nabla \times)^2 \bm{v} )_\Omega + ( \nabla p_0, \bm{v} )_\Omega = (\bm{f}, \bm{v})_\Omega. $$
        Moreover, \ref{est on test v} has included that $ \nabla \cdot \bm{u}_0 = 0 $. Therefore, $ ( \bm{u}_0, p_0 ) $ is the weak solution.
\end{proof}
\begin{remark}
    \begin{enumerate}
        \item Notice that the convergence is derived from the compact embedding, which can never deliver the convergence rate. This is an essential limit of this strategy.
        \item Our conclusion involving the convergence under the $ H (\text{curl};\Omega) $ norm naturally raises the question of convergence in $ H (\text{curl}^2;\Omega) $. This is unavailable though, because strong convergence of the second order derivative is usually derived from the weakly convergence of the third order derivative, while there is no such conclusion that $ \nabla \times A (\nabla \times)^2 \bm{u} $ admits any regularity. In fact, this strategy derives nothing more than the boundedness of $ H (\text{curl}^2;\Omega) $ error, which is a direct consequence of the energy estimate.
        \item While the present work assumes that $ A $ is piecewise constant into polyhedral subdomains, the conclusions above should also govern the characteristic of the generalized scenario:
        \begin{itemize}
            \item $ \{\Omega_i\}_{i\in\Lambda} $ is a partition of $\Omega$ such that each boundary $ \partial \Omega_i $ is Lipschitz except for a set of zero two dimensional Hausdorff measure ($\mathcal{H}^{2}$).
            \item $ A $ is a matrix-valued function that is $C^1$ on each of $\Omega_i$, whose evaluation is uniformly positive definite.
        \end{itemize}
        The concrete proof requires more effort though, and is a future goal of us.
    \end{enumerate}
\end{remark}

\section{The Optimal Convergence Rate}
\subsection{Additional Assumptions}
The convergence rate conclusion, different from the former theoretical convergence, requires more on the mesh and the regularity assumption. Throughout this section, we assume that $ \mathcal{T}_h $ are subordinate to $\{\Omega_i\}_{i\in\Lambda}$, and the following assumptions hold:
\begin{assumption}\label{assumption}
    On each region of the partition $\{\Omega_i\}_{i\in\Lambda}$, $\exists s^i_0$, $s^i_1$, $s^i_2$, $q^i_3$, and $s^i_p$, such that
    \begin{equation*}
        \begin{split}
            \bm{u} \in H^{s^i_0}(\Omega)^3,\\
            \nabla\times \bm{u} \in H^{s^i_1}(\Omega)^3,\\
            A(\nabla\times\nabla\times \bm{u} )\in H^{s^i_2}(\Omega)^3,\\
            \nabla\times A(\nabla\times\nabla\times \bm{u} )\in L^{q^i_3}(\Omega)^3,\\
            p\in H^{s^i_p}(\Omega).
        \end{split}
    \end{equation*}
    Here we allow distinct evaluations of $s^i_0$, $s^i_1$, $s^i_2$, $q^i_3$, and $s^i_p$ on different regions $\Omega_i$, or even on each distinct open subset of $ \Omega $. However, we assume uniform bounds of them respectively: $s^i_0>\frac{1}{2},\;s^i_1\ge1,\;s^i_2\in (0,\frac{1}{2})$, $q^i_3\in (\frac{6}{3+2s_2},2)$, and $s^i_p\ge1$.
\end{assumption}
Whenever the context is clear, we omit the superscripts and write $s_0$, $s_1$,$s_2$, $q_3$, and $s_p$. 
\begin{remark}
    The assumption resembles the one in \cite{dong2024c}. Notice that the regularity of $ \nabla\times A(\nabla\times\nabla\times \bm{u}) $ never reaches $L^2$, so that the assumption is weaker than $ A(\nabla\times)^2 \bm{u} \in H^{\frac{1}{2}+\delta}(\Omega) $, preserving the novelty. However, there is still a gap between the current assumption and the former unconditional convergence result. What if $ \nabla\times A(\nabla\times\nabla\times \bm{u}) $ only exists in the distributional sense, instead of having an integrable function representation? Analysis on this case encounters essential difficulties, and we have not overcome them yet. It is our future goal to drop the restriction on third order derivative.
\end{remark}
Now we employ the argument established in \cite{ern2022quasi}, specifying an enlarged space endowed with a proper space to do analysis:
\begin{definition}\label{aug norm}
    We define
    \begin{equation*}
        \begin{split}
            V_s=\{ \bm{v} \in H(\text{curl}^2;\Omega) \cap H(\text{div }0;\Omega) | A(\nabla\times)^2 \bm{v} \in H^{s_2}(\Omega),\\
        \nabla\times A(\nabla\times)^2 \bm{v} \in L^{q_3}(\Omega),\\
        (\nabla\times)^2 A(\nabla\times)^2 \bm{v} \in L^2(\Omega)\},
        \end{split}
    \end{equation*}
    denote $ E_\# = V_s + E_h $, and endow it with the \textit{augmented norm}
    \begin{equation*}
        \begin{aligned}
            \| \bm{v}_h \|_{E_\#}:=&(\| \bm{v}_h \|_{E_h}^2 + \sum_{K\in \mathcal{T}_h} h_K^{2s_2} | A(\nabla\times)^2 \bm{v}_h |_{H^{s_2}(K)}^2 + h_K^{5-\frac{6}{q_3}} \|\nabla\times A(\nabla\times)^2 \bm{v}_h \|_{L^{q_3}(K)}^2\\
            & + h_K^{4} \|(\nabla\times)^2 A(\nabla\times)^2 \bm{v}_h \|_{L^2(K)}^2)^{1/2}.
        \end{aligned}
    \end{equation*}
\end{definition}
Note that assumption \ref{assumption} implies that the weak solution $ \bm{u} $ lands in $ V_s $. As we described in the introduction, the augmented norm $\|\cdot\|_{E_\#}$ assembles quantities computed on each cell and interface. The purpose of this section is to establish the error estimate between the numerical solutions, $ \bm{u}_h $, and the weak solution, $u$, under this augmented norm $\|\cdot\|_{E_\#}$. \par

\subsection{An Abstract Framework}
To reveal the quasi-optimal framework, it will be beneficial to demonstrate it in the following abstract form. In this subsection, the notation is all \textit{ad hoc}, only standing for abstractly imposed spaces and norms, but does not need to be the same as the settings before. We consider the following saddle point problem: Let $E$, $Q$ be Banach spaces. Find $(u,p)\in E\times Q$ such that
\begin{equation}
    \begin{aligned}
        a(u,v)+b(p,v)=(f,v)\quad&\forall v\in E;\\
        b(q,u)=0,\quad&\forall q\in Q.
    \end{aligned}
\end{equation}
where $f\in E^*$. Denote $V=\{v\in E|b(q,v)=0,\forall q\in Q\}$, and on assuming the $V$-coercive of $a$ and the inf-sup condition of $b$, then the Babu\v{s}ka–Brezzi Theorem implies well-posedness of this problem. Assume that the weak solution $u\in V_s\subset V$.\par
We impose the following discretizations: Find $( u_h ,p_h) \in E_h \times Q_h$ such that
\begin{equation}
    \begin{aligned}
        a_h( u_h , v_h )+b_h(p_h, v_h )=(f_h, \bm{v}_h ) \quad & \forall v_h \in E_h;\\
        b_h(q_h, u_h )=0,\quad&\forall q_h\in Q_h.
    \end{aligned}
\end{equation}
Define $E_{\#}=V_s+E_h$, $ V_h =\{ v_h \in E_h| b_h(q_h, v_h )=0, \forall q_h\in Q_h\} $, and $\delta_h: V_h \to V_h ^*$ as $\langle\delta_h( v_h ), w_h \rangle=a_h( u_h - v_h , w_h )$. Assume that $Q_h\subset Q$ and $b_h=b$ on $Q_h\times E_\#$.
\begin{theorem}\label{abs quasi}
    If the following estimates hold:
    \begin{enumerate}[(a)]
        \item $\|a_h\|_{E_h\times E_h}:=C_A<\infty$; $\inf_{ v_h \in V_h }\frac{a_h( v_h , v_h )}{\| v_h \|_{E_h}^2}:=\alpha>0$; 
        \item $\|b\|_{Q_h\times E_h}:=C_B<\infty$; $\inf_{q_h\in Q_h}\sup_{ w_h \in E_h}\frac{b_h(q_h, w_h )}{\|q_h\|_{Q}\| w_h \|_{E_h}}:=\beta>0$; 
        \item $\sup_{ v_h \in V_h }\frac{\| v_h \|_{E_\#}}{\| v_h \|_{E_h}}:=c_\#<\infty$; 
        \item $\sup_{ v_h \in V_h }\frac{\|\delta_h( v_h )\|_{ V_h ^*}}{\|u- v_h \|_{E_\#}+\|p-p_h\|_Q}:=\omega_\#<\infty$;
    \end{enumerate}
    then the numerical solution $( u_h ,p_h)$ satisfies
    \begin{equation*}
        \begin{aligned}
            \|u- u_h \|_{E_\#}&\le(1+\frac{c_\#\omega_\#}{\alpha})(\inf_{ v_h \in V_h }\|u- v_h \|_{E_\#}+\|p-p_h\|_Q);\\
            \|p-p_h\|_Q&\le(1+\frac{C_B}{\beta})\inf_{q_h\in Q_h}\|p-q_h\|_Q.
        \end{aligned}
    \end{equation*}
\end{theorem}
\begin{proof}
    The existence of the numerical solution follows from the Babu\v{s}ka–Brezzi Theorem. For an arbitrary $ v_h \in V_h $,
    \begin{align*}
         \|u- u_h \|_{E_\#}&\le\|u- v_h \|_{E_\#}+\| v_h - u_h \|_{E_\#}\\
         &\le\|u- v_h \|_{E_\#}+c_\#\| v_h - u_h \|_{E_h}\\
         &\le\|u- v_h \|_{E_\#}+\frac{c_\#}{\alpha}\sup_{ w_h \in V_h }\frac{|a_h( v_h - u_h , w_h )|}{\| w_h \|_{E_h}}\\
         &=\|u- v_h \|_{E_\#}+\frac{c_\#}{\alpha}\|\delta_h( v_h )\|_{ V_h ^*}\\
         &\le(1+\frac{c_\#\omega_\#}{\alpha})(\|u- v_h \|_{E_\#}+\|p-p_h\|_Q).
    \end{align*}
    For any $q_h\in Q_h$,
    \begin{align*}
         \|p-p_h\|_Q&\le\|p-q_h\|_Q+\|p_h-q_h\|_Q\\
         &\le\|p-q_h\|_{E_\#}+\frac{1}{\beta}\sup_{ w_h \in E_h}\frac{|b_h(p_h-q_h, w_h )|}{\| w_h \|_{E_h}}\\
         &\le(1+\frac{C_B}{\beta})\|p-q_h\|_Q.
    \end{align*}
\end{proof}

\subsection{The Quasi-Optimal Argument}
Now we prove that the quasi-optimal result claimed above holds for \ref{ip mtd}. It suffices to check the four assumptions in Theorem \ref{abs quasi}.\par
First, we notice that the following estimates have been covered in the proof of \ref{well-posed}:
    \begin{enumerate}[(a)]
        \item $ \exists C_A < \infty $, $ \alpha > 0 $, $\forall \bm{v}_h , \bm{w}_h \in V_h $, 
        $$|a_h( \bm{v}_h , \bm{w}_h )|\le C_A\| \bm{v}_h \|_{E_h}\| \bm{w}_h \|_{E_h}; \quad |a_h( \bm{v}_h , \bm{v}_h )|\ge\alpha\| \bm{v}_h \|_{E_h}^2. $$
        \item $ \exists C_B < \infty $, $ \beta > 0 $, $ \forall q_h \in Q_h$, 
        $$|b_h(q_h, \bm{w}_h )|\le C_B\|q_h\|_{H^1_0(\Omega)}\| \bm{w}_h \|_{E_h}; \quad \sup_{ \bm{w}_h \in E_h}\frac{b_h(q_h, \bm{w}_h )}{\| \bm{w}_h \|_{E_h}}>\beta \|q_h\|_{H^1_0(\Omega)}$$
    \end{enumerate}
Then, since the mesh $\mathcal{T}_h$ and the finite element space $E_h$ we are applying in this work surely satisfy the assumptions of \ref{inv ineq}, we have
\begin{theorem}
    $\exists c_\#>0$, $\sup_{ \bm{v}_h \in V_h }\| \bm{v}_h \|_{E_\#}\le c_\#\| \bm{v}_h \|_{E_h}$.
\end{theorem}
\begin{proof}
    It is sufficient to find $ c_\# $ such that
    \[\sum_{K\in\mathcal{T}_h}h_K^{s_2}|A(\nabla\times\nabla\times  \bm{v}_h )|_{H^{s_2}(K)}\le c_\#\sum_{K\in\mathcal{T}_h}\|A(\nabla\times\nabla\times  \bm{v}_h )\|_{L^2(K)}, \]
    \[\sum_{K\in\mathcal{T}_h}h_K^{\frac{5}{2}-\frac{3}{q_3}}\|\nabla\times A(\nabla\times\nabla\times  \bm{v}_h )\|_{L^{q_3}(K)}\le c_\#\sum_{K\in\mathcal{T}_h}\|A(\nabla\times\nabla\times  \bm{v}_h )\|_{L^2(K)}, \]
    and 
    \[\sum_{K\in\mathcal{T}_h}h_K^{2}\|\nabla\times\nabla\times A(\nabla\times\nabla\times  \bm{v}_h )\|_{L^2(K)}\le c_\#\sum_{K\in\mathcal{T}_h}\|A(\nabla\times\nabla\times  \bm{v}_h )\|_{L^2(K)} \]
    hold for any $ \bm{v}_h  \in V_h $, which are all direct consequences of \ref{inv ineq}.
\end{proof}
The rest of this section aims evaluating $\omega_\#$, which requires the most efforts. We apply the face-to-cell lifting technique established as in \cite[Lemma 3.1]{ern2022quasi}.
\begin{lemma}
    $\forall \frac{3}{2}<q<2$, $\exists L_{F}^{K}$ to be a bounded operator from $W^{1-\frac{1}{q},q}(F)$ to $W^{1,q}(K)$ where $K\in\mathcal{T}_h$, $F\in\partial K$, such that
    \begin{enumerate}
        \item the trace map $\gamma:W^{1,q}(K)\to W^{1-\frac{1}{q},q}(\partial K)$ satisfies
        \begin{equation}\label{lift1}
            \gamma\circ L^K_F(\phi)=
            \begin{cases}
                \phi\quad\text{on }F;\\
                0\quad\text{on }\partial K\backslash F;\\
            \end{cases}
        \end{equation}
        \item $\forall\phi\in W^{1-\frac{1}{q},q}(F)$, the following estimate holds:
        \begin{equation}\label{lift2}
            h_K^{\frac{6}{q}-3}\|L^K_F(\phi)\|_{L^\frac{q}{q-1}(K)}+h_K\|\nabla L^K_F(\phi)\|_{L^q(K)}\le Ch_K^{\frac{3}{q}}h_F^{-\frac{2}{q}}(\|\phi\|_{L^q(F)}+h_F^{1-\frac{1}{q}}|\phi|_{W^{1-\frac{1}{q},q}(F)}).
        \end{equation}
    \end{enumerate}
\end{lemma}
\begin{proof}
    We start by constructing the required operator on the reference element. Let $\hat{K}$ be the standard simplex in $\mathbb{R}^3$, and $\hat{F}$ be a face of it. \cite[ Corollary 1.4.4.5]{grisvard2011elliptic} claims that the zero extension from $\hat{F}$ to $\partial\hat{K}$, $\tilde{\cdot}:W^{1-\frac{1}{q},q}( \hat{F})\to W^{1-\frac{1}{q},q}(\partial \hat{K})$ such that
    \begin{equation*}
        \tilde{\phi}=
        \begin{cases}
            \phi\quad\text{on }F\\
            0\quad\text{on }\partial K\backslash F,
        \end{cases}
    \end{equation*}
    is a bounded linear operation in $\mathcal{L}(W^{1-\frac{1}{q},q}( \hat{F}), W^{1-\frac{1}{q},q}(\partial \hat{K}))$. In addition, \cite[Theorem 1.5.1.3]{grisvard2011elliptic} claims that the trace map $\gamma:W^{1,q}(\hat{K})\to W^{1-\frac{1}{q},q}(\partial \hat{K})$ admits a continuous right inverse: 
    \begin{equation*}
        L_{\partial \hat{K}}^{\hat{K}}:\mathcal{L}(W^{1-\frac{1}{q},q}(\partial \hat{K}),W^{1,q}(\hat{K}))\text{ such that }\gamma\circ L_{\partial \hat{K}}^{\hat{K}}=Id.
    \end{equation*}
    Hence, we can define $L_{\hat{F}}^{\hat{K}}$ as the composition of the two:
    \begin{equation*}
        L_{\hat{F}}^{\hat{K}}(\phi)=L_{\partial\hat{F}}^{\hat{K}}(\tilde{\phi}),\;\forall\phi\in W^{1,q}(\hat{F}).
    \end{equation*}
    \begin{enumerate}
        \item Composing properties of the two operators, condition \ref{lift1} holds automatically.
        \item We establish the result for $K\in\mathcal{T}_h$. Obviously there exists an affine map $T_K:\hat{K}\to K$ such that $\|B_K^{-1}\|\ge\sigma\|B_K\|$. Therefore, by the scaling argument we have
        \begin{equation*}
            \|v\|_{L^q(K)}+h_K\|\nabla v\|_{L^q(K)}\le Ch_K^{\frac{3}{q}}\|v\circ T_K\|_{L^q(\hat{K})}+\|\hat{\nabla}(v\circ T_K)\|_{L^q(\hat{K})},\;\forall v\in W^{1,q}(K)
        \end{equation*}
        and 
        \begin{equation*}
            \|\phi\circ T_K\|_{L^q(\hat{F})}+|\phi\circ T_K|_{W^{1-\frac{1}{q},q}(\hat{F})}\le Ch_F^{-\frac{2}{q}}(\|\phi\|_{L^q(F)}+h_F^{1-\frac{1}{q}}|\phi|_{W^{1-\frac{1}{q},q}(F)}),\;\forall\phi\in W^{1-\frac{1}{q},q}(F).
        \end{equation*}
        Besides, since $q\in(\frac{3}{2},2)$, we have $\frac{3q}{3-q}>\frac{q}{q-1}$, and the Sobolev Embedding Theorem gives
        \begin{equation*}
            h_K^{\frac{6}{q}-3}\|v\|_{L^\frac{q}{q-1}(K)}\le C (\|v\|_{L^q(K)}+h_K\|\nabla v\|_{L^q(K)}),\;\forall v\in W^{1,q}(K)
        \end{equation*}
        We define $L^K_F$ as $L^K_F(\phi)=L_{\hat{F}}^{\hat{K}}(\phi\circ T_K)\circ T_K^{-1}$, and apply the estimates above to derive \ref{lift2}.
    \end{enumerate}
\end{proof}
With this face-to-cell lifting, we can define the following bilinear form:
\begin{definition}
    We specify the lifting operator $L^K_F$ constructed from $W^{1-\frac{1}{q_3},q_3}(F)$ to $W^{1,q_3}(K)$, and define $ n_\#: E_\# \times E_h \to \mathbb{R}$ such that for any $ (\bm{w}, \bm{v}_h) \in E_\# \times E_h $,
    \begin{equation*}
        \begin{split}
        n_\#( \bm{w}, \bm{v}_h )=\sum_{F\in\mathcal{F}_h}\sum_{K\in\mathcal{T}_F}\frac{\epsilon_{K,F}}{2}(( A(\nabla\times)^2 \bm{w}|_K,\nabla\times L_F^K([\![\nabla\times  \bm{v}_h  ]\!]))_K\\
        -\langle\nabla\times A(\nabla\times)^2 \bm{w}|_K,L_F^K([\![\nabla\times  \bm{v}_h  ]\!])\rangle_K).
    \end{split}
    \end{equation*}
\end{definition}

\begin{theorem}
    $n_\#$ is a well-defined continuous bilinear form on $E_\#\times E_h$ with the estimate
    \begin{equation*}
        |n_\#(\bm{w}, \bm{v}_h )|\le C\| \bm{w} \|_{E_\#}\| \bm{v}_h \|_{E_h}.
    \end{equation*}
\end{theorem}
\begin{proof}
    First, since $s_2\in(0,\frac{1}{2})$, $q_3\in(\frac{3}{2},2)$, which means that $L^K_F:W^{1-\frac{1}{q_3},q_3}(F)\to W^{1,q_3}(K)$ is well defined. \par
    Next, notice that for any $ \bm{v}_h \in E_h$, $[\![\nabla\times  \bm{v}_h  ]\!]$ is a polynomial on $F$, and shall have a vanishing normal component. Thus, it lives in both spaces: $W^{1-\frac{1}{q_3},q_3}(F)$ and $L^2(F)$, and by \ref{lift2} and the discrete inverse inequality,
    \begin{align*}
        &h_K^{\frac{3}{q_3}-\frac{5}{2}}\|L^K_F([\![\nabla\times  \bm{v}_h ]\!])\|_{L^{\frac{q_3}{q_3-1}}(K)}+h_K^{\frac{3}{2}-\frac{3}{q_3}}\|\nabla\times L^K_F([\![\nabla\times  \bm{v}_h ]\!])\|_{L^{q_3}(K)}\\
        &\quad\le h_K^{\frac{1}{2}-\frac{3}{q_3}}Ch_K^{\frac{3}{q_3}}h_F^{-\frac{2}{q_3}}(\|[\![\nabla\times  \bm{v}_h ]\!]\|_{L^{q_3}(F)}+h_F^{1-\frac{1}{q_3}}|[\![\nabla\times  \bm{v}_h ]\!]|_{W^{1-\frac{1}{q_3},q_3}(F)})\\
        &\quad\le Ch_K^\frac{1}{2}h_F^{-1}\|[\![\nabla\times  \bm{v}_h ]\!]\|_{L^{2}(F)}\\
        &\quad= Ch_K^\frac{1}{2}h_F^{-1}\| \bm{n}_F\times[\![\nabla\times  \bm{v}_h ]\!]\|_{L^{2}(F)}\\
        &\quad\le C(\frac{\tau}{h_F})^\frac{1}{2}\| \bm{n}_F\times[\![\nabla\times  \bm{v}_h ]\!]\|_{L^{2}(F)}.
    \end{align*}
    \par
    Then, for any $ \bm{w} \in E_\#$, by definition $ \bm{w} $ admits a decomposition $ \bm{w} = \bm{w}_s+ \bm{w}_h $, where $ \bm{w}_s\in V_s $ and $ \bm{w}_h \in E_h$. Hence, $ A(\nabla\times)^2 \bm{w}|_K\in H^{s_2}(K)$ and $ \nabla\times A(\nabla\times)^2 \bm{w}|_K\in L^{q_3}(K)$. Notice that $\frac{q_3}{q_3-1}<\frac{6}{3-2s_2}$, and by the Sobolev Embedding Theorem, $A(\nabla\times)^2 \bm{w}|_K\in L^{\frac{q_3}{q_3-1}}(K)$, together with the estimate
    \begin{equation*}
        \|A(\nabla\times)^2 \bm{w} \|_{L^{\frac{q_3}{q_3-1}}(K)} \le C h_K^{\frac{3}{2}-\frac{3}{q_3}} (\| A(\nabla\times)^2 \bm{w} \|_{L^{2}(K)} + h_K^{s_2}| A(\nabla\times)^2 \bm{w} |_{H^{s_2}(K)}).
    \end{equation*}
    \par
    Therefore, both parts of $n_\#(\bm{w}, \bm{v}_h )$ can be regarded as assembling the pairings between $L^{\frac{q_3}{q_3-1}}(K)$ and $ L^{q_3}(K)$ for $ K \in  \mathcal{T}_h $. Moreover, we have the estimate
    \begin{align*}
        |n_\#(\bm{w}, \bm{v}_h )|\le & C\sum_{K\in\mathcal{T}_h}\sum_{F\in\partial K}\frac{1}{2}(\| A(\nabla\times)^2 \bm{w} \|_{L^{\frac{q_3}{q_3-1}}(K)}\|\nabla\times L^K_F([\![\nabla\times  \bm{v}_h ]\!])\|_{L^{q_3}(K)}\\
        & + \|\nabla\times A(\nabla\times)^2 \bm{w}\|_{L^{q_3}(K)}\|L^K_F([\![\nabla\times  \bm{v}_h ]\!])\|_{L^{\frac{q_3}{q_3-1}}(K)})\\
        \le & C \sum_{K\in\mathcal{T}_h}\sum_{F\in\partial K}(h_K^{\frac{3}{2}-\frac{3}{q_3}}(\|A(\nabla\times)^2 \bm{w}\|_{L^{2}(K)}+h_K^{s_2}|A(\nabla\times)^2 \bm{w} |_{H^{s_2}(K)})\\
        & \qquad \cdot \|\nabla\times L^K_F([\![\nabla\times \bm{v}_h ]\!])\|_{L^{q_3}(K)}\\
        & + h_K^{\frac{5}{2}-\frac{3}{q_3}}\|\nabla\times A(\nabla\times)^2 \bm{w} \|_{L^{q_3}(K)}h_K^{\frac{3}{q_3}-\frac{5}{2}}\|L^K_F([\![\nabla\times \bm{v}_h ]\!])\|_{L^{\frac{q_3}{q_3-1}}(K)})\\
        \le & C \sum_{K\in\mathcal{T}_h}(\| A(\nabla\times)^2 \bm{w} \|_{L^{2}(K)}+h_K^{s_2}| A(\nabla\times)^2 \bm{w} |_{H^{s_2}(K)} + h_K^{\frac{5}{2}-\frac{3}{q_3}}\|\nabla\times A(\nabla\times)^2 \bm{w}\|_{L^{q_3}(K)})\cdot\\
        & \sum_{K\in\mathcal{T}_h}\sum_{F\in\partial K}(h_K^{\frac{3}{q_3}-\frac{5}{2}}\|L^K_F([\![\nabla\times  \bm{v}_h ]\!])\|_{L^{\frac{q_3}{q_3-1}}(K)}+h_K^{\frac{3}{2}-\frac{3}{q_3}}\|\nabla\times L^K_F([\![\nabla\times  \bm{v}_h ]\!])\|_{L^{q_3}(K)}),
    \end{align*}    
    in which the quantity
    \begin{align*}
        \sum_{K\in\mathcal{T}_h}\| A(\nabla\times)^2 \bm{w} \|_{L^{2}(K)} + h_K^{s_2}| A(\nabla\times)^2 \bm{w} |_{H^{s_2}(K)} + h_K^{\frac{5}{2}-\frac{3}{q_3}}\|\nabla\times A(\nabla\times)^2 \bm{w} \|_{L^{q_3}(K)}
    \end{align*}
    is a component of $\| \bm{w} \|_{E_\#}$, and we have shown that 
    \begin{align*}
        &h_K^{\frac{3}{q_3}-\frac{5}{2}}\|L^K_F([\![\nabla\times  \bm{v}_h ]\!])\|_{L^{\frac{q_3}{q_3-1}}(K)}+h_K^{\frac{3}{2}-\frac{3}{q_3}}\|\nabla\times L^K_F([\![\nabla\times  \bm{v}_h ]\!])\|_{L^{q_3}(K)}\\
        &\quad\le C(\frac{\tau}{h_F})^\frac{1}{2}\| \bm{n}_F\times[\![\nabla\times  \bm{v}_h ]\!]\|_{L^{2}(F)},
    \end{align*}
    a component of $\| \bm{v}_h \|_{E_h}$. In conclusion,
    \begin{equation*}
        |n_\#( \bm{w}, \bm{v}_h )|\le C\|\bm{w}\|_{E_\#}\| \bm{v}_h \|_{E_h},\;\forall w\in E_\#,\; \bm{v}_h \in E_h.
    \end{equation*}
\end{proof}
Next, we illustrate a useful fact.
\begin{lemma}\label{density}
    For any $\phi\in L^2(\Omega)$ such that $\nabla\times\phi\in L^{q}(\Omega)$, we can raise a sequence $\{\psi_n\}_n$ in $C^\infty_c(\Omega)^3$ such that 
    \begin{equation*}
        \lim_{n\to\infty}\|\psi_n-\phi\|_{L^2(\Omega)}=0,\;\lim_{n\to\infty}\|\nabla\times\psi_n-\nabla\times \phi\|_{L^{q}(\Omega)}=0,
    \end{equation*}
    simultaneously. 
\end{lemma}
\begin{proof}
    Since $\Omega$ is Lipschitz, we only need to reveal this conclusion on a Lipschitz hypograph
    \begin{equation*}
        \Omega'=\{x\in\mathbb{R}^3|x_3<\zeta(x')\},
    \end{equation*}
    where $x'=(x_1,x_2)$ and $\zeta$ is a Lipschitz function on $\mathbb{R}^2$.\par
    The first step is to determine, for an arbitrary $\varepsilon>0$ and $\phi$ as in the setting, some $\delta>0$ such that for
    \begin{equation*}
        \phi_\delta(x)=\phi(x',x_3+\delta),
    \end{equation*}
    the two relationships hold simultaneously:
    \begin{equation*}
    \begin{split}
         \|\phi-\phi_\delta\|_{L^2(\Omega')}<\frac{\varepsilon}{4};\\
         \|\nabla\times \phi-\nabla\times \phi_\delta\|_{L^q(\Omega')}<\frac{\varepsilon}{4}.
    \end{split}
    \end{equation*}
    The first term holds for $\delta$ is sufficiently small, since $\phi$ is absolutely continuous, and
    \begin{equation*}
        \|\phi-\phi_\delta\|_{L^2(\Omega')}\to0,\;\delta\to0.
    \end{equation*}
    For the second part, we notice that $\nabla\times (\phi_\delta)=(\nabla\times \phi)_\delta$, and by the absolute continuity,
    \begin{equation*}
        \|\nabla\times \phi-\nabla\times \phi_\delta\|_{L^q(\Omega')}<\frac{\varepsilon}{4}
    \end{equation*}
    when $\delta$ is sufficiently small.\par
    The second step is to raise some $\psi\in C^\infty_c(\Omega)$ such that $\|\phi-\psi\|_{H^r(\Omega')}+\|\nabla\times \phi-\nabla\times \psi\|_{L^q(\Omega')}<\varepsilon$. Doing convolution with the standard mollifier with the radius of the support less that $\delta$, we construct $\psi\in C^\infty(\mathbb{R}^3)$ such that 
    \begin{equation*}
    \begin{split}
         \|\phi_\delta-\psi\|_{L^2(\mathbb{R}^3)}<\frac{\varepsilon}{4};\\
         \|\nabla\times\phi_\delta-\nabla\times \psi\|_{L^q(\mathbb{R}^3)}<\frac{\varepsilon}{4}.
    \end{split}
    \end{equation*}
    Moreover, $\text{supp}\psi\subset\text{supp}\phi|_\delta+\frac{\delta}{2}\subset\Omega$, i.e., $\psi\in C^\infty_c(\Omega)$. Hence we conclude
    \begin{align*}
        &\|\phi-\psi\|_{L^2(\Omega')}+\|\nabla\times\phi-\nabla\times \psi\|_{L^q(\Omega')} \\
        &\le \|\phi-\phi_\delta\|_{L^2(\Omega')}+\|\nabla\times \phi-\nabla\times \phi_\delta\|_{L^q(\Omega')}\\
        &\quad+ \|\phi_\delta-\psi\|_{L^2(\mathbb{R}^3)}+\|\nabla\times \phi_\delta-\nabla\times \psi\|_{L^q(\mathbb{R}^3)}\\
        &<\varepsilon.
    \end{align*}
\end{proof}
With this lemma we can prove the following two crucial identities:
\begin{theorem}
    Restricting the performance of $n_\#$ on two subspaces of $E_\#$, we have
    \begin{enumerate}
        \item $\forall  \bm{w}_h \in E_h$, 
        \begin{equation}\label{bili id 1}
            n_\#( \bm{w}_h , \bm{v}_h )=\sum_{F\in\mathcal{F}_h}\langle \{\!\{A(\nabla\times\nabla\times  \bm{w}_h )\}\!\}, \bm{n}_F\times L_F^K([\![\nabla\times  \bm{v}_h  ]\!])\rangle_F;
        \end{equation}
        \item $\forall w\in V_s$, 
        \begin{equation}\label{bili id 2}
            n_\#(\bm{w}, \bm{v}_h )=\sum_{K\in\mathcal{T}_h}(A(\nabla\times)^2 \bm{w} ,(\nabla\times)^2 \bm{v}_h )_K-((\nabla\times)^2 A(\nabla\times)^2 \bm{w}, \bm{v}_h )_K.
        \end{equation}
    \end{enumerate}
\end{theorem}
\begin{proof}
    \begin{enumerate}
        \item For $ \bm{w}_h \in E_h$, we can integrate by part to derive
        \begin{align*}
            &(A(\nabla\times\nabla\times  \bm{w}_h ),\nabla\times L_F^K([\![\nabla\times  \bm{v}_h  ]\!]))_K-( \nabla\times A(\nabla\times\nabla\times  \bm{w}_h ),L_F^K([\![\nabla\times  \bm{v}_h  ]\!]))_K\\
            &=\langle A(\nabla\times\nabla\times  \bm{w}_h )|_K, \bm{n}_K\times L_F^K([\![\nabla\times  \bm{v}_h  ]\!])\rangle_{\partial K}\\
            &=\varepsilon_{K,F}\langle A(\nabla\times\nabla\times  \bm{w}_h )|_K, \bm{n}_F\times L_F^K([\![\nabla\times  \bm{v}_h  ]\!])\rangle_{F}+0, 
        \end{align*}
        since $ L_F^K([\![\nabla\times  \bm{v}_h  ]\!] = 0 $ on $ \partial K \backslash F $. Therefore, 
        \begin{align*}
            &\sum_{F\in\mathcal{F}_h}\sum_{K\in\mathcal{T}_F}\frac{\epsilon_{K,F}}{2}(A(\nabla\times\nabla\times  \bm{w}_h ),\nabla\times L_F^K([\![\nabla\times  \bm{v}_h  ]\!]))_K\\
            &-( \nabla\times A(\nabla\times\nabla\times  \bm{w}_h ),L_F^K([\![\nabla\times  \bm{v}_h  ]\!]))_K\\
            &=\sum_{F\in\mathcal{F}_h}\sum_{K\in\mathcal{T}_F}\frac{1}{2}\langle A(\nabla\times\nabla\times  \bm{w}_h )|_K, \bm{n}_F\times L_F^K([\![\nabla\times  \bm{v}_h  ]\!])\rangle_{F}\\
            &=\sum_{F\in\mathcal{F}_h}\langle \{\!\{A(\nabla\times\nabla\times  \bm{w}_h )\}\!\}, \bm{n}_F\times L_F^K([\![\nabla\times  \bm{v}_h  ]\!])\rangle_{F}, \quad\forall  \bm{v}_h \in V_h .
        \end{align*}
        \item We start by illustrating \ref{bili id 2} for any $\sigma\in C^\infty_c(\Omega)^3$:
        \begin{equation*}
            \begin{aligned}
                \sum_{F\in\mathcal{F}_h}\sum_{K\in\mathcal{T}_F}&\frac{\epsilon_{K,F}}{2}(\langle\sigma,\nabla\times L_F^K([\![\nabla\times  \bm{v}_h  ]\!])\rangle_K-(\nabla\times\sigma,L_F^K([\![\nabla\times  \bm{v}_h  ]\!]))_K\\
                &=(\sigma,\nabla\times\nabla\times  \bm{v}_h )_\Omega-(\nabla\times\nabla\times \sigma, \bm{v}_h )_\Omega, \quad\forall  \bm{v}_h \in V_h .
            \end{aligned}
        \end{equation*}
        This can be done by the following: 
        \begin{align*}
            &\sum_{F\in\mathcal{F}_h}\sum_{K\in\mathcal{T}_F}\frac{\epsilon_{K,F}}{2}(\langle\sigma,\nabla\times L_F^K([\![\nabla\times  \bm{v}_h  ]\!])\rangle_K-(\nabla\times\sigma,L_F^K([\![\nabla\times  \bm{v}_h  ]\!]))_K\\
            &=\sum_{F\in\mathcal{F}_h}\sum_{K\in\mathcal{T}_F}\frac{\epsilon_{K,F}}{2}\langle\sigma, \bm{n}_K\times L_F^K([\![\nabla\times  \bm{v}_h  ]\!])\rangle_{\partial K}\\
            &=\sum_{F\in\mathcal{F}_h}\sum_{K\in\mathcal{T}_F}\frac{\epsilon_{K,F}}{2}\int_F\sigma\cdot  \bm{n}_K\times[\![\nabla\times  \bm{v}_h  ]\!]ds+0\\
            &=\sum_{F\in\mathcal{F}_h}\int_F\sigma\cdot  \bm{n}_F\times(\nabla\times  \bm{v}_h |_{K^+}-\nabla\times  \bm{v}_h |_{K^-})ds\\
             &=\sum_{F\in\mathcal{F}_h}\int_F\sigma\cdot(  \bm{n}_{K^+}\times\nabla\times  \bm{v}_h |_{K^+})ds+\int_F\sigma\cdot(  \bm{n}_{K^-}\times\nabla\times  \bm{v}_h |_{K^-})ds\\
            &=\sum_{K\in\mathcal{T}_h}\int_{\partial K}\sigma\cdot(  \bm{n}_K\times\nabla\times  \bm{v}_h |_K)ds\\
            &=\sum_{K\in\mathcal{T}_h}(\sigma,\nabla\times\nabla\times  \bm{v}_h )_K-(\nabla\times \sigma,\nabla\times  \bm{v}_h )_K\\
            &=\sum_{K\in\mathcal{T}_h}(\sigma,\nabla\times\nabla\times  \bm{v}_h )_K-(\nabla\times\nabla\times \sigma, \bm{v}_h )_K\\
            &=(\sigma,\nabla\times\nabla\times  \bm{v}_h )_\Omega-(\nabla\times\nabla\times \sigma, \bm{v}_h )_\Omega.
        \end{align*}
        Next, we verify this for an arbitrary $\bm{w}\in V_s$. We raise a sequence $\{\sigma_n\}_n$ in $C^\infty_c(\Omega)^3$ such that 
        \begin{equation*}
            \lim_{n\to\infty}\|\sigma_n-A(\nabla\times)^2 \bm{w}\|_{L^2(\Omega)}=0,\;\lim_{n\to\infty}\|\nabla\times\sigma_n-\nabla\times A(\nabla\times)^2 \bm{w}\|_{L^{q_3}(\Omega)}=0,
        \end{equation*}
        simultaneously. Therefore, take $\sigma_n$ in \ref{bili id 2} and let $ \bm{n} \to \infty $. By Lemma \ref{density}, each side of \ref{bili id 2} converges to the corresponding term in the desired identity about $\bm{w}$, respectively. Since the identity holds for each $\sigma_n$, the conclusion is solid for $\bm{w}$.
    \end{enumerate}
\end{proof}
Utilizing these two facts, we accomplish the estimate of $\omega_\#$:
\begin{theorem}
    For $ \bm{v}_h $, $a_h$, $\|\cdot\|_{E_h}$, and $\|\cdot\|_{E_\#}$ defined as in \ref{ip mtd},  $\exists \omega_\#$, $\forall  \bm{v}_h \in V_h $, 
    \begin{equation*}
        \sup_{ \bm{w}_h \in V_h }\frac{|a_h( \bm{u}_h - \bm{v}_h , \bm{w}_h )|}{\| \bm{w}_h \|_{ \bm{w}_h }}\le\omega_\#(\| \bm{u} - \bm{v}_h \|_{E_\#}+\|p-p_h\|_{H_0^1(\Omega)}),
    \end{equation*}
    where $ \bm{u}_h $ and $ \bm{u} $ are the discrete and exact solution, respectively.
\end{theorem}
\begin{proof}
    We choose $ \bm{w}_h \in V_h $ arbitrarily. First, notice that $ \bm{u} \in V_s $ enables \ref{bili id 2}, and $ (\nabla\times)^2 A(\nabla\times)^2 \bm{u} = \bm{f}-\nabla p $ in the $ [L^2(\Omega)]^3 $ sense. Thus,
    \begin{align*}
        & \sum_{K\in\mathcal{T}_h}(A(\nabla\times)^2 \bm{u}, (\nabla\times)^2 \bm{w}_h )_{K} -  \bm{n}_{\#}( \bm{u}, \bm{w}_h ) -(\bm{f}, \bm{w}_h )_\Omega + (\nabla p, \bm{w}_h )_\Omega\\
        & = ((\nabla\times)^2 A(\nabla\times)^2 \bm{u}, \bm{w}_h )_\Omega - (\bm{f}, \bm{w}_h )_\Omega + (\nabla p, \bm{w}_h )_\Omega\\
        & = 0.
    \end{align*}
    Nevertheless, \ref{ip mtd} derives
    \begin{align*}
        a_h( \bm{u}_h - \bm{v}_h , \bm{w}_h )=(\bm{f}, \bm{w}_h )_\Omega-(\nabla p_h, \bm{w}_h )_\Omega-a_h( \bm{v}_h , \bm{w}_h )
    \end{align*}
    Summing the two identities up and using \ref{bili id 1},
    \begin{align*}
        & a_h( \bm{u}_h - \bm{v}_h , \bm{w}_h ) + 0 \\
        & = (\nabla p-\nabla p_h, \bm{w}_h )_\Omega - a_h( \bm{v}_h , \bm{w}_h ) -  \bm{n}_{\#}(\bm{u}, \bm{w}_h ) + \sum_{K\in\mathcal{T}_h} (A(\nabla\times)^2 \bm{u},(\nabla\times)^2 \bm{w}_h )_{K} \\
        & = (\nabla p-\nabla p_h, \bm{w}_h )_\Omega- \bm{n}_{\#}( \bm{u} - \bm{v}_h , \bm{w}_h ) + \sum_{K\in\mathcal{T}_h}(A(\nabla\times)^2( \bm{u} - \bm{v}_h ),(\nabla\times)^2 \bm{w}_h )_{K}\\
        & \quad-\sum_{F\in\mathcal{F}_h}\langle  \bm{n}_F\times[\![\nabla\times( \bm{u} - \bm{v}_h ) ]\!],\{\!\{A(\nabla\times\nabla\times \bm{w}_h )\}\!\}\rangle_{F}\\
        & \quad +\sum_{F\in\mathcal{F}_h}\frac{\tau}{h_F}\langle  \bm{n}_F\times[\![\nabla\times( \bm{u} - \bm{v}_h ) ]\!], \bm{n}_F\times[\![\nabla\times \bm{w}_h ]\!]\rangle_{F}
    \end{align*}
    Thus, by the boundedness of $n_\#$ and the triangular inequality,
    \begin{equation*}
        |a_h( \bm{u}_h - \bm{v}_h , \bm{w}_h )|\le C(\| \bm{u} - \bm{v}_h \|_{E_\#}+\|p-p_h\|_{H^1_0(\Omega)})\| \bm{w}_h \|_{E_h}.
    \end{equation*}
\end{proof}
So far we have achieved estimates that are vital for applying the abstract framework. 
\begin{theorem}
    \begin{equation*}
        \begin{aligned}
            \| \bm{u} - \bm{u}_h \|_{E_\#}&\le C(\inf_{ \bm{v}_h \in V_h }\| \bm{u} - \bm{v}_h \|_{E_\#}+\|p-p_h\|_{H^1_0(\Omega)});\\
            \|p-p_h\|_{H^1_0(\Omega)}&\le C\inf_{q_h\in Q_h}\|p-q_h\|_{H^1_0(\Omega)}.
        \end{aligned}
    \end{equation*}
\end{theorem}

\subsection{Polynomial Approximation}
To accomplish the ultimate estimates, we need to study the approximating property of the finite elements.
\begin{theorem}\label{Nedelec est 3}
    \begin{equation*}
    \begin{aligned}
        \inf_{ \bm{v}_h \in V_h }\| \bm{u} - \bm{v}_h \|_{E\#}\le& C\sum_{K\in\mathcal{T}_h}(h_K^{s_0}\| \bm{u} \|_{H^{s_0}(K)}+h_K^{s_1-1}\| \nabla \times \bm{u} \|_{H^{s_1}(K)}+h_K^{s_2}\| A(\nabla\times)^2 \bm{u} \|_{H^{s_2}(K)}\\
        &+h_K^{\frac{5}{2}-\frac{3}{q_3}}\|\nabla\times A(\nabla\times)^2 \bm{u}_h \|_{L^{q_3}(K)}+h_K^2\| (\nabla\times)^2 A(\nabla\times)^2 \bm{u} \|_{L^2(K)}).
    \end{aligned}
    \end{equation*}
\end{theorem}
\begin{proof}
    To raise the approximation in $ V_h $, we need to implement a variety of the second N\'{e}d\'{e}lec interpolation, $\Pi^{\text{curl}}_{h,k}$, so that it falls in $ V_h $. We apply $ \Pi^E_h \bm{u} $ as raised in (3.4), \cite{chen2021analysis}:
    \[\Pi^E_h \bm{u} =\Pi^{\text{curl}}_{h,k} \bm{u} +\nabla\sigma_h\]
    where $\sigma_h\in Q_h$ such that 
    \[(\nabla\sigma_h,\nabla q_h)=(\bm{u} -\Pi^{\text{curl}}_{h,k} \bm{u} ,\nabla q_h),\forall q_h\in Q_h.\]
    Lemma 3.1, (5.9), and (5.15) of \cite{chen2021analysis} show that
    \begin{equation*}
        \| \bm{u} -\Pi^E_h \bm{u}\|_{E_h}\le C(h^{s_0}\| \bm{u} \|_{H^{s_0}(\Omega)} + h^{s_1-1}\|\nabla \times \bm{u} \|_{H^{s_1}(\Omega)}),
    \end{equation*} 
    remaining the following parts:
    \begin{enumerate}
        \item $\sum_{K\in\mathcal{T}_h}h_K^{s_2}|A (\nabla\times)^2 (\bm{u} -\Pi^{\text{curl}}_{h,k} \bm{u} )|_{H^{s_2}(K)}$;
        \item $\sum_{K\in\mathcal{T}_h}h_K^{\frac{5}{2}-\frac{3}{q_3}}\|\nabla\times A(\nabla\times)^2 (\bm{u} -\Pi^{\text{curl}}_{h,k} \bm{u} ))\|_{L^{q_3}(K)}$;
        \item $\sum_{K\in\mathcal{T}_h}h_K^{2}\|(\nabla\times)^2 A (\nabla\times)^2 (\bm{u} -\Pi^{\text{curl}}_{h,k} \bm{u} )\|_{L^2(K)}$.
    \end{enumerate}
    We shall control them in a routine manner: establish estimates element-wisely, and sum them up. Notice that the error estimate derived in this form will gain the local property as we promised.
    \begin{enumerate}
        \item \textit{Estimation of }$h_K^{s_2}|A\nabla\times\nabla\times(\bm{u} -\Pi^{\text{curl}}_{h,k} \bm{u} )|_{H^{s_2}(K)}$. \par
        As cited before, $\nabla\times\Pi^{\text{curl}}_{k} \bm{u} = \Pi^{BDM}_{k-1}(\nabla\times \bm{u})$. Hence
        \begin{align*}
            &h_K^{s_2}|A(\nabla\times\nabla\times( \bm{u} -\Pi^{\text{curl}}_{k} \bm{u}))|_{H^{s_2}(K)}\\
            &= h_K^{s_2}|A(\nabla\times(I-\Pi^{BDM}_{k-1})(\nabla\times \bm{u} )|_{H^{s_2}(K)}\\
            &\le h_K^{s_2}|(I-\Pi^{L^2}_{0})A(\nabla\times\nabla\times \bm{u} )|_{H^{s_2}(K)}+h_K^{s_2}|(\Pi^{L^2}_{0}-\Pi^{L^2}_{k-2})A(\nabla\times\nabla\times \bm{u} )|_{H^{s_2}(K)}\\
            &\quad +h_K^{s_2}|(\Pi^{L^2}_{k-2}A(\nabla\times\nabla\times \bm{u} )-A(\nabla\times\Pi^{L^2}_{k-1}(\nabla\times \bm{u} ))|_{H^{s_2}(K)}\\
            &\quad+h_K^{s_2}|A(\nabla\times(\Pi^{L^2}_{k-1}-\Pi^{BDM}_{k-1})(\nabla\times \bm{u} ))|_{H^{s_2}(K)}\\
            &:=I_1+I_2+I_3+I_4
        \end{align*}
        where $\Pi^{L^2}_{l}$ are the $L^2$ projections to $\mathcal{P}_{l}(K)$, $l=0,k-2,k-1$.\par
        First, since $\mathcal{P}_0$ are the constant functions on the element $K$, their $s_2$ semi-norms are all 0, and
        \begin{equation*}
            I_1=h_K^{s_2}|A(\nabla\times\nabla\times \bm{u} )|_{H^{s_2}(K)}.
        \end{equation*}
        \par
        Next, since $\mathcal{P}_0\subset\mathcal{P}_{k-1}$, the combination of the construction of $\Pi^{L^2}_{k-1}$, the discrete inverse inequality, and the Bramble-Hilbert lemma derives
        \begin{align*}
            I_2&\le C\|(\Pi^{L^2}_{0}-\Pi^{L^2}_{k-2})A(\nabla\times\nabla\times \bm{u} )\|_{L^2(K)}\\
            &=C\|\Pi^{L^2}_{k-2}\circ(\Pi^{L^2}_{0}-I)A(\nabla\times\nabla\times \bm{u} )\|_{L^2(K)}\\
            &\le C\|(\Pi^{L^2}_{0}-I)A(\nabla\times\nabla\times \bm{u} )\|_{L^2(K)}\\
            &\le Ch_K^{s_2}|A(\nabla\times\nabla\times \bm{u} )|_{H^{s_2}(K)}.        
        \end{align*}
        Then, we notice that $A\nabla\times\Pi^{L^2}_{k-1}(\nabla\times \bm{u} )\in\mathcal{P}_{k-2}(K)$, and carry the same technique:
        \begin{align*}
            I_3&\le C\|\Pi^{L^2}_{k-2}A(\nabla\times\nabla\times \bm{u} )-A(\nabla\times\Pi^{L^2}_{k-1}(\nabla\times \bm{u} ))\|_{L^2(K)}\\
            &=C\|\Pi^{L^2}_{k-2}A(\nabla\times(I-\Pi^{L^2}_{k-1})(\nabla\times \bm{u} ))\|_{L^2(K)}\\
            &\le C\|A(\nabla\times(I-\Pi^{L^2}_{k-1})(\nabla\times \bm{u} ))\|_{L^2(K)}\\
            &\le Ch_K^{-1}\|(I-\Pi^{L^2}_{k-1})(\nabla\times \bm{u} )\|_{L^2(K)}\\
            &\le Ch_K^{s_1-1}|\nabla\times \bm{u} |_{H^{s_1}(K)}.        
        \end{align*}
        Last, we apply \ref{RT est}:
        \begin{align*}
            I_4&\le C h_K^{-1}\|(\Pi^{L^2}_{k-1}-\Pi^{BDM}_{k-1})(\nabla\times \bm{u} )\|_{L^2(K)}\\
            &=h_K^{-1}\|\Pi^{L^2}_{k-1}(I-\Pi^{BDM}_{k-1})(\nabla\times \bm{u} )\|_{L^2(K)}\\
            &\le Ch_K^{-1}\|(I-\Pi^{BDM}_{k-1})(\nabla\times \bm{u} )\|_{L^2(K)}\\
            &\le Ch_K^{s_1-1}|(\nabla\times \bm{u} )|_{H^{s_1}(K)}.     
        \end{align*}
        Consequentially,
        \begin{equation*}
            h_K^{s_2}|A\nabla\times\nabla\times(\bm{u} -\Pi^{\text{curl}}_{h,k} \bm{u} )|_{H^{s_2}(K)}\le C(h_K^{s_1-1}|\nabla \times u|_{H^{s_1}(K)}+h_K^{s_2}|A(\nabla\times\nabla\times \bm{u} )|_{H^{s_2}(K)}).
        \end{equation*}
        \item \textit{Estimation of }$h_K^{\frac{5}{2}-\frac{3}{q_3}}\|\nabla\times A(\nabla\times\nabla\times (\bm{u} -\Pi^{\text{curl}}_{h,k} \bm{u} ))\|_{L^{q_3}(K)}$. Note that
        \begin{align*}
            &h_K^{\frac{5}{2}-\frac{3}{q_3}}\|\nabla\times A(\nabla\times\nabla\times (\bm{u} -\Pi^{\text{curl}}_{h,k} \bm{u} ))\|_{L^{q_3}(K)}\\
            &\le h_K^{\frac{5}{2}-\frac{3}{q_3}}\|\nabla\times(I-\Pi^{L^2}_{0}) A(\nabla\times\nabla\times \bm{u} )\|_{L^{q_3}(K)}\\
            &\quad+h_K^{\frac{5}{2}-\frac{3}{q_3}}\|\nabla\times(\Pi^{L^2}_{0}-\Pi^{L^2}_{k-2})A(\nabla\times\nabla\times \bm{u} )\|_{L^{q_3}(K)}\\
            &\quad+h_K^{\frac{5}{2}-\frac{3}{q_3}}\|\nabla\times \Pi^{L^2}_{k-2}(A(\nabla\times\nabla\times \bm{u} ))-\nabla\times A(\nabla\times\Pi^{BDM}_{k-1}(\nabla\times \bm{u} ))\|_{L^{q_3}(K)}\\
            &=I_1+I_2+I_3.
        \end{align*}
        By the same strategy as before,
        \begin{equation*}
            I_1=h_K^{\frac{5}{2}-\frac{3}{q_3}}\|\nabla\times A(\nabla\times\nabla\times \bm{u} )\|_{L^{q_3}(K)};
        \end{equation*}
        \begin{align*}
            I_2&\le Ch_K^{\frac{5}{2}-\frac{3}{q_3}}h_K^{-1+\frac{3}{q_3}-\frac{3}{2}}\|(\Pi^{L^2}_{0}-\Pi^{L^2}_{k-2})A(\nabla\times\nabla\times \bm{u} )\|_{L^{2}(K)}\\
            &=C\|\Pi^{L^2}_{k-2}(\Pi^{L^2}_{0}-I)A(\nabla\times\nabla\times \bm{u} )\|_{L^{2}(K)}\\
            &\le C\|(\Pi^{L^2}_{0}-I)A(\nabla\times\nabla\times \bm{u} )\|_{L^{2}(K)}\\
            &\le Ch_K^{s_2}|A(\nabla\times\nabla\times \bm{u} )|_{H^{s_2}(K)};
        \end{align*}
        \begin{align*}
            I_3\le & C\|\Pi^{L^2}_{k-2}(A(\nabla\times\nabla\times \bm{u} ))-A(\nabla\times\Pi^{BDM}_{k-1}(\nabla\times \bm{u} ))\|_{L^{2}(K)}\\
            \le & Ch_K^{s_1-1}|\nabla\times \bm{u} |_{H^{s_1}(K)}.
        \end{align*}
        Consequentially,
        \begin{equation*}
        \begin{aligned}
            & h_K^{\frac{5}{2}-\frac{3}{q_3}}\|\nabla\times A(\nabla\times\nabla\times (\bm{u} -\Pi^{\text{curl}}_{h,k} \bm{u} ))\|_{L^{q_3}(K)}\\
            &\le C (h_K^{s_1-1}|\nabla\times \bm{u} |_{H^{s_1}(K)}+h_K^{s_2}|A(\nabla\times\nabla\times \bm{u} )|_{H^{s_2}(K)}+h_K^{\frac{5}{2}-\frac{3}{q_3}}\|\nabla\times A(\nabla\times\nabla\times \bm{u} )\|_{L^{q_3}(K)}.
        \end{aligned}
        \end{equation*}
        \item \textit{Estimation of }$ h_K^{2}\|\nabla\times\nabla\times A\nabla\times\nabla\times(\bm{u} -\Pi^{\text{curl}}_{h,k} \bm{u} )\|_{L^2(K)}$. We have 
        \begin{align*}
            &h_K^{2}\|\nabla\times\nabla\times A(\nabla\times\nabla\times (\bm{u} -\Pi^{\text{curl}}_{h,k} \bm{u} ))\|_{L^{2}(K)}\\
            &\le h_K^{2}\|\nabla\times\nabla\times(I-\Pi^{L^2}_{0}) A(\nabla\times\nabla\times \bm{u} )\|_{L^{2}(K)}\\
            &\quad+h_K^{2}\|\nabla\times\nabla\times(\Pi^{L^2}_{0}-\Pi^{L^2}_{k-2})A(\nabla\times\nabla\times \bm{u} )\|_{L^{2}(K)}\\
            &\quad+h_K^{2}\|\nabla\times\nabla\times \Pi^{L^2}_{k-2}(A(\nabla\times\nabla\times \bm{u} ))-\nabla\times\nabla\times A(\nabla\times\Pi^{BDM}_{k-1}(\nabla\times \bm{u} ))\|_{L^{2}(K)}\\
            &=I_1+I_2+I_3.
        \end{align*}
        Similarly,
        \begin{equation*}
            I_1=h_K^{2}\|\nabla\times\nabla\times A(\nabla\times\nabla\times \bm{u} )\|_{L^{2}(K)};
        \end{equation*}
        \begin{align*}
            I_2&\le C\|(\Pi^{L^2}_{0}-\Pi^{L^2}_{k-2})A(\nabla\times\nabla\times \bm{u} )\|_{L^{2}(K)}\\
            &\le Ch_K^{s_2}|A(\nabla\times\nabla\times \bm{u} )|_{H^{s_2}(K)};
        \end{align*}
        \begin{align*}
            I_3\le & C\|\Pi^{L^2}_{k-2}(A(\nabla\times\nabla\times \bm{u} ))-A(\nabla\times\Pi^{BDM}_{k-1}(\nabla\times \bm{u} ))\|_{L^{2}(K)}\\
            \le & Ch_K^{s_1-1}|\nabla\times \bm{u} |_{H^{s_1}(K)}.
        \end{align*}
        Consequentially,
        \begin{equation*}
        \begin{aligned}
            & h_K^{2}\|\nabla\times\nabla\times A(\nabla\times\nabla\times (\bm{u} -\Pi^{\text{curl}}_{h,k} \bm{u} ))\|_{L^{2}(K)}\\
            &\le C (h_K^{s_1-1}|\nabla\times \bm{u} |_{H^{s_1}(K)}+h_K^{s_2}|A(\nabla\times\nabla\times \bm{u} )|_{H^{s_2}(K)}+h_K^{2}\|\nabla\times\nabla\times A(\nabla\times\nabla\times \bm{u} )\|_{L^{2}(K)}.
        \end{aligned}
        \end{equation*}
    \end{enumerate}    
\end{proof}
\begin{theorem} 
    \begin{equation*}
        \begin{aligned}
            &\| \bm{u}- \bm{u}_h \|_{E_\#}+\|p-p_h\|_{H^1_0(\Omega)}\\
            &\le C\sum_{K\in\mathcal{T}_h}(h_K^{s_0}\|\bm{u}\|_{H^{s_0}(K)}+h_K^{s_1-1}\|\nabla \times \bm{u}\|_{H^{s_1}(K)} + h_K^{s_2}\|A(\nabla\times\nabla\times \bm{u} )\|_{H^{s_2}(K)}\\
            & + h_K^{\frac{5}{2}-\frac{3}{q_3}}\|\nabla\times A(\nabla\times\nabla\times \bm{u} )\|_{L^{q_3}(K)} + h_K^2\|\nabla\times\nabla\times A(\nabla\times\nabla\times \bm{u} ))\|_{L^2(K)}\\
            & + h_K^{s_p}\|p\|_{H^{s_p}(K)}).
        \end{aligned}
    \end{equation*}
\end{theorem}

\section*{Acknowledgment}
The author would like to appreciate Prof. Weifeng QIU for providing invaluable guidance during the preparation of this paper, especially his abundant knowledge, profound insight and innovative thinking on this research topic.

\bibliographystyle{plain}
\bibliography{reference}

@book{grisvard2011elliptic,
  title={Elliptic problems in nonsmooth domains},
  author={Grisvard, Pierre},
  year={2011},
  publisher={SIAM}
}

@article{ern2022quasi,
  title={Quasi-optimal nonconforming approximation of elliptic PDEs with contrasted coefficients and ${H^{1+ r}}$, $r>0$, regularity},
  author={Ern, Alexandre and Guermond, Jean-Luc},
  journal={Foundations of Computational Mathematics},
  volume={22},
  number={5},
  pages={1273--1308},
  year={2022},
  publisher={Springer}
}

@article{chen2021analysis,
  title={Analysis of an interior penalty DG method for the quad-curl problem},
  author={Chen, Gang and Qiu, Weifeng and Xu, Liwei},
  journal={IMA Journal of Numerical Analysis},
  volume={41},
  number={4},
  pages={2990--3023},
  year={2021},
  publisher={Oxford University Press}
}

@article{qiu2023enriched,
  title={An enriched Ciarlet-Raviart scheme for the biharmonic equation},
  author={Qiu, Weifeng},
  journal={Communications on Analysis and Computation},
  volume={1},
  number={1},
  pages={1--11},
  year={2023},
  publisher={Communications on Analysis and Computation}
}

@article{ern2017finite,
  title={Finite element quasi-interpolation and best approximation},
  author={Ern, Alexandre and Guermond, Jean-Luc},
  journal={ESAIM: Mathematical Modelling and Numerical Analysis},
  volume={51},
  number={4},
  pages={1367--1385},
  year={2017},
  publisher={EDP Sciences}
}

@article{nedelec1980mixed,
  title={Mixed finite elements in $\mathbb{R}^3$},
  author={N{\'e}d{\'e}lec, Jean-Claude},
  journal={Numerische Mathematik},
  volume={35},
  number={3},
  pages={315--341},
  year={1980},
  publisher={Springer}
}

@article{nedelec1986new,
  title={A new family of mixed finite elements in $\mathbb{R}^3$},
  author={N{\'e}d{\'e}lec, Jean-Claude},
  journal={Numerische Mathematik},
  volume={50},
  number={1},
  pages={57--81},
  year={1986},
  publisher={Springer}
}

@book{monk2003finite,
  title={Finite element methods for Maxwell's equations},
  author={Monk, Peter},
  year={2003},
  publisher={Oxford university press}
}

@article{alonso1999optimal,
  title={An optimal domain decomposition preconditioner for low-frequency time-harmonic Maxwell equations},
  author={Alonso, Ana and Valli, Alberto},
  journal={Mathematics of Computation},
  volume={68},
  number={226},
  pages={607--631},
  year={1999}
}

@book{boffi2013mixed,
  title={Mixed finite element methods and applications},
  author={Boffi, Daniele and Brezzi, Franco and Fortin, Michel and others},
  volume={44},
  year={2013},
  publisher={Springer}
}

@article{dong2022hybrid,
  title={Hybrid high-order and weak Galerkin methods for the biharmonic problem},
  author={Dong, Zhaonan and Ern, Alexandre},
  journal={SIAM Journal on Numerical Analysis},
  volume={60},
  number={5},
  pages={2626--2656},
  year={2022},
  publisher={SIAM}
}

@article{dong2024c,
  title={${C^0}$-hybrid high-order methods for biharmonic problems},
  author={Dong, Zhaonan and Ern, Alexandre},
  journal={IMA Journal of Numerical Analysis},
  volume={44},
  number={1},
  pages={24--57},
  year={2024},
  publisher={Oxford University Press}
}

@article{monk2012finite,
  title={Finite element methods for Maxwell's transmission eigenvalues},
  author={Monk, Peter and Sun, Jiguang},
  journal={SIAM Journal on Scientific Computing},
  volume={34},
  number={3},
  pages={B247--B264},
  year={2012},
  publisher={SIAM}
}

@article{haddar2004interior,
  title={The interior transmission problem for anisotropic Maxwell's equations and its applications to the inverse problem},
  author={Haddar, Houssem},
  journal={Mathematical methods in the applied sciences},
  volume={27},
  number={18},
  pages={2111--2129},
  year={2004},
  publisher={Wiley Online Library}
}

@article{hiptmair2002finite,
  title={Finite elements in computational electromagnetism},
  author={Hiptmair, Ralf},
  journal={Acta Numerica},
  volume={11},
  pages={237--339},
  year={2002},
  publisher={Cambridge University Press}
}

@article{weber1980local,
  title={A local compactness theorem for Maxwell's equations},
  author={Weber, Ch and Werner, P},
  journal={Mathematical Methods in the Applied Sciences},
  volume={2},
  number={1},
  pages={12--25},
  year={1980},
  publisher={Wiley Online Library}
}

@article{zheng2011nonconforming,
  title={A nonconforming finite element method for fourth order curl equations in $\mathbb{R}^3$},
  author={Zheng, Bin and Hu, Qiya and Xu, Jinchao},
  journal={Mathematics of computation},
  volume={80},
  number={276},
  pages={1871--1886},
  year={2011}
}

@article{nicaise2018singularities,
  title={Singularities of the quad curl problem},
  author={Nicaise, Serge},
  journal={Journal of Differential Equations},
  volume={264},
  number={8},
  pages={5025--5069},
  year={2018},
  publisher={Elsevier}
}

@article{guermond2003mixed,
  title={Mixed finite element approximation of an MHD problem involving conducting and insulating regions: the 3D case},
  author={Guermond, JL and Minev, PD},
  journal={Numerical Methods for Partial Differential Equations: An International Journal},
  volume={19},
  number={6},
  pages={709--731},
  year={2003},
  publisher={Wiley Online Library}
}

@article{costabel1991coercive,
  title={A coercive bilinear form for Maxwell's equations},
  author={Costabel, Martin},
  journal={Journal of mathematical analysis and applications},
  volume={157},
  number={2},
  pages={527--541},
  year={1991},
  publisher={Elsevier}
}

@article{zhang2009family,
  title={A family of 3D continuously differentiable finite elements on tetrahedral grids},
  author={Zhang, Shangyou},
  journal={Applied Numerical Mathematics},
  volume={59},
  number={1},
  pages={219--233},
  year={2009},
  publisher={Elsevier}
}

@article{zhang2019h,
  title={${H (curl^{2})}$-conforming finite elements in 2 dimensions and applications to the quad-curl problem},
  author={Zhang, Qian and Wang, Lixiu and Zhang, Zhimin},
  journal={SIAM Journal on Scientific Computing},
  volume={41},
  number={3},
  pages={A1527--A1547},
  year={2019},
  publisher={SIAM}
}

@article{hong2012discontinuous,
  title={A discontinuous Galerkin method for the fourth-order curl problem},
  author={Hong, Qingguo and Hu, Jun and Shu, Shi and Xu, Jinchao},
  journal={Journal of Computational Mathematics},
  pages={565--578},
  year={2012},
  publisher={JSTOR}
}

@article{sun2016mixed,
  title={A mixed FEM for the quad-curl eigenvalue problem},
  author={Sun, Jiguang},
  journal={Numerische Mathematik},
  volume={132},
  number={1},
  pages={185--200},
  year={2016},
  publisher={Springer}
}

@article{zhang2020family,
  title={A family of curl-curl conforming finite elements on tetrahedral meshes},
  author={Zhang, Qian and Zhang, Zhimin},
  journal={CSIAM Trans. Appl. Math},
  volume={1},
  number={4},
  pages={639--663},
  year={2020}
}

@article{veeser2018quasiI,
  title={Quasi-optimal nonconforming methods for symmetric elliptic problems. I---Abstract theory},
  author={Veeser, Andreas and Zanotti, Pietro},
  journal={SIAM Journal on Numerical Analysis},
  volume={56},
  number={3},
  pages={1621--1642},
  year={2018},
  publisher={SIAM}
}

@article{veeser2018quasiIII,
  title={Quasi-optimal nonconforming methods for symmetric elliptic problems. III---Discontinuous Galerkin and other interior penalty methods},
  author={Veeser, Andreas and Zanotti, Pietro},
  journal={SIAM Journal on Numerical Analysis},
  volume={56},
  number={5},
  pages={2871--2894},
  year={2018},
  publisher={SIAM}
}
\end{document}